\documentclass[twoside,a4paper,reqno,11pt]{amsart} 
\usepackage[top=28mm,right=28mm,bottom=28mm,left=28mm]{geometry}

\usepackage{amsfonts, amsmath, amssymb, mathrsfs, bm, latexsym, stmaryrd, array, hyperref, mathtools, bold-extra}

\renewcommand{\a}{\alpha}
\renewcommand{\b}{\beta}

\newcommand{\e}{\varepsilon}
\renewcommand{\l}{\lambda}

\renewcommand{\O}{\Omega}

\newcommand{\normeq}{\trianglelefteqslant}

\renewcommand{\to}{\rightarrow}

\newcommand{\leqs}{\leqslant}
\newcommand{\geqs}{\geqslant}

\newcommand{\vs}{\vspace{2mm}}

\newcommand{\what}{\widehat}

\makeatletter
\newcommand{\imod}[1]{\allowbreak\mkern4mu({\operator@font mod}\,\,#1)}
\makeatother

\newtheorem{theorem}{Theorem} 
\newtheorem{theoremm}{Theorem}

\newtheorem{corr}[theoremm]{Corollary}

\newtheorem*{conj*}{Conjecture}

\newtheorem{conj}[theorem]{Conjecture}

\newtheorem{thm}{Theorem}[section] 
 
\newtheorem{lem}[thm]{Lemma}

\theoremstyle{definition}
\newtheorem{rem}[thm]{Remark}
\newtheorem{remk}[theorem]{Remark}

\begin{document}

\author{Timothy C. Burness}
\address{T.C. Burness, School of Mathematics, University of Bristol, Bristol BS8 1UG, UK}
\email{t.burness@bristol.ac.uk}

\author{Hong Yi Huang}
\address{H.Y. Huang, School of Mathematics and Statistics, University of St Andrews, KY16 9SS, UK}
\curraddr{Department of Mathematics, Southern University of Science and Technology, Shenzhen 518055, Guangdong, P. R. China}
\email{11612012@mail.sustech.edu.cn}
 
\title[On the intersections of nilpotent subgroups in simple groups]{On the intersections of nilpotent subgroups \\ in simple groups} 

\begin{abstract}
Let $G$ be a finite group and let $H_p$ be a Sylow $p$-subgroup of $G$. A recent conjecture of Lisi and Sabatini asserts the existence of an element $x \in G$ such that $H_p \cap H_p^x$ is inclusion-minimal in the set $\{H_p \cap H_p^g \,:\, g \in G\}$ for every prime $p$. 
For a simple group $G$, in view of a theorem of Mazurov and Zenkov from 1996, the conjecture implies the existence of an element $x \in G$ with $H_p \cap H_p^x = 1$ for all $p$. In turn, this statement implies a conjecture of Vdovin from 2002, which asserts  that if $G$ is simple and $H$ is a nilpotent subgroup, then $H \cap H^x = 1$ for some $x \in G$. 

In this paper, we adopt a probabilistic approach to prove the Lisi-Sabatini conjecture for all non-alternating simple groups. By combining this with earlier work of Kurmazov on nilpotent subgroups of alternating groups, we complete the proof of Vdovin's conjecture. Moreover, by combining our proof with earlier work of Zenkov on alternating groups, we are able to establish a stronger form of Vdovin's conjecture: if $G$ is simple and $A,B$ are nilpotent subgroups, then $A \cap B^x = 1$ for some $x \in G$. To obtain these results, we study the probability that a random pair of Sylow $p$-subgroups in a simple group of Lie type intersect trivially, complementing recent work of Diaconis et al. and Eberhard on symmetric and alternating groups.
\end{abstract}

\date{\today}

\maketitle

\section{Introduction}\label{s:intro}

Let $G$ be a finite group and let $H$ be a nilpotent subgroup of $G$. By a theorem of Zenkov \cite{Z_3}, there exist elements $x,y \in G$ such that 
\[
H \cap H^x \cap H^y \leqs F(G),
\]
where $F(G)$ is the Fitting subgroup of $G$. This result is optimal in the sense that there exist examples $H < G$ with the property that the intersection of any two conjugates of $H$ is not contained in $F(G)$ (for instance, if $G = S_8$ and $H$ is a Sylow $2$-subgroup of $G$, then $H \cap H^x \ne 1$ for all $x \in G$). Zenkov's theorem is just one of many results in a substantial literature on the intersections of nilpotent subgroups of finite groups, which can be traced all the way back to work of Passman \cite{Pass} in the 1960s.

In this paper, we are interested in the special case where $G$ is a (non-abelian) finite simple group. Since the Fitting subgroup of $G$ is trivial, Zenkov's theorem implies that for any nilpotent subgroup $H$ we have
$H \cap H^x \cap H^y = 1$ for some $x,y \in G$. In fact, Vdovin has proposed the following conjecture, which is stated as Problem 15.40 in the \emph{Kourovka Notebook} \cite{Kou}. Here and throughout the paper, whenever we refer to a simple group, we implicitly assume the group is non-abelian. 

\begin{conj}[Vdovin \cite{Kou}, 2002]\label{c:nilp}
Let $G$ be a finite simple group and let $H$ be a nilpotent subgroup of $G$. Then $H \cap H^x = 1$ for some $x \in G$.
\end{conj}

This conjecture has been resolved for alternating groups \cite{Kurm} and sporadic groups \cite{Z_spor}. In addition, there are some partial results for certain low rank groups of Lie type (for example, see \cite{Z_L2} for the $2$-dimensional linear groups ${\rm L}_2(q)$, and \cite{Z_dim3} for ${\rm L}_3(q)$ and ${\rm U}_3(q)$). But in general, the conjecture remains open for groups of Lie type.

An important special case of Conjecture \ref{c:nilp} was established in earlier work of Mazurov and Zenkov \cite{ZM}.

\begin{theorem}[Mazurov \& Zenkov \cite{ZM}, 1996]\label{t:mz}
Let $G$ be a finite simple group and let $H$ be a Sylow $p$-subgroup of $G$. Then $H \cap H^x = 1$ for some $x \in G$.
\end{theorem}

In \cite{ZM}, the proof of Theorem \ref{t:mz} for groups of Lie type rests on a representation-theoretic result of Green \cite{Green}, which states that if a $p$-block for a finite group $G$ has defect group $D$ and $H$ is a Sylow $p$-subgroup of $G$, then $H \cap H^x = D$ for some $x \in G$. This can then be combined with later work of Michler \cite{Michler} and Willems \cite{Willems} in the 1980s, who showed that every finite simple group of Lie type has a $p$-block with trivial defect group, for every prime divisor $p$ of the order of the group.

\begin{remk}\label{r:as}
Recall that a finite group $G$ is \emph{almost simple} if $G_0 \normeq G \leqs {\rm Aut}(G_0)$ for some non-abelian simple group $G_0$, where we identify $G_0$ with its group of inner automorphisms. It turns out that Theorem \ref{t:mz} does not extend to almost simple groups and we refer the reader to our recent paper \cite{BH}, where we extend earlier work of Zenkov to determine all the pairs $(G,H)$, where $G$ is almost simple, $H$ is a Sylow $p$-subgroup and $H \cap H^x \ne 1$ for all $x \in G$. This of course includes the example highlighted above, where $G = S_8$ and $H$ is a Sylow $2$-subgroup.
\end{remk}

\begin{remk}\label{r:conn}
The statements of Conjecture \ref{c:nilp} and Theorem \ref{t:mz} can be rephrased in several different ways, providing natural connections to other well studied problems. 

\begin{itemize}\addtolength{\itemsep}{0.2\baselineskip}
\item[{\rm (a)}] Let $G$ be a finite group and let $H$ be a core-free subgroup, which allows us to view $G$ as a transitive permutation group $G \leqs {\rm Sym}(\O)$ on the set $\O = G/H$ of cosets of $H$ in $G$. Then the \emph{base size} of $G$, denoted $b(G,H)$, is defined to be the minimal size of a subset of $\O$ with trivial pointwise stabiliser in $G$. This is a fundamental and intensively studied invariant in permutation group theory, with an extensive literature stretching all the way back to the 19th century. In the language of bases, Zenkov's main theorem in \cite{Z_3} reveals that $b(G,H) \leqs 3$ for every finite transitive group $G$ with $F(G) = 1$ and nilpotent point stabiliser $H$. And similarly, Conjecture \ref{c:nilp} asserts that $b(G,H) = 2$ for every finite simple transitive group $G$ with a nilpotent point stabiliser $H \ne 1$.  

\item[{\rm (b)}] The previous observation connects with another well-known conjecture of Vdovin, which is stated as Problem 17.41(b) in \cite{Kou}. This asserts that $b(G,H) \leqs 5$ for every transitive permutation group $G$ with trivial soluble radical and soluble point stabiliser $H$ (the example $G = S_8$ and $H = S_4 \wr S_2$ shows that the purported bound is best possible). This conjecture is proved for primitive groups in \cite{Bur21}, and Vdovin gives a reduction of the general problem to almost simple groups in \cite{Vdovin}. Further work \cite{Bay1,B23} reduces the problem to groups of Lie type, and Baykalov \cite{Bay2,Bay3} has made recent progress towards a proof for the classical groups. But the general problem remains open, including the special case where $G$ is a simple group.

\item[{\rm (c)}] It is also worth highlighting a connection with the notion of the \emph{depth} of a subgroup $H$ of a finite group $G$, denoted $d_G(H)$, which was introduced in a 2011 paper by Boltje, Danz and K\"{u}lshammer \cite{BDK}. This is a positive integer defined in terms of the inclusion of complex group algebras $\mathbb{C}H \subseteq \mathbb{C}G$ and there has been a focus on studying the depth of subgroups of simple and almost simple groups (see \cite{B25} and the references therein). The relevant connection arises from the observation that if $H$ is core-free and $b(G,H) = 2$, then $d_G(H) = 3$ (see \cite[Theorem 6.9]{BKK}). So in the language of subgroup depth, Conjecture \ref{c:nilp} asserts that every non-trivial nilpotent subgroup of a simple group has depth $3$.
\end{itemize}
\end{remk}

A renewed interest in Conjecture \ref{c:nilp} stems from a recent paper of Lisi and Sabatini \cite{LS}, where they propose the following far reaching generalisation. In the statement, we write $\pi(G)$ for the set of distinct prime divisors of $|G|$.

\begin{conj}[Lisi \& Sabatini \cite{LS}, 2025]\label{c:ls_main}
Let $G$ be a finite group with $\pi(G) = \{p_1, \ldots, p_k\}$ and let $P_i$ be a Sylow $p_i$-subgroup of $G$. Then there exists an element $x \in G$ such that for each $i$, $P_i \cap P_i^x$ is inclusion-minimal in the set $\{P_i \cap P_i^g \,:\, g \in G\}$.
\end{conj}

In \cite{LS}, the conjecture is proved for all metanilpotent groups of odd order, as well as all sufficiently large alternating and symmetric groups. However, the conjecture remains wide open in general; in particular, it is open for groups of odd order.

In view of Theorem \ref{t:mz}, Conjecture \ref{c:ls_main} takes the following striking form for simple groups. 

\begin{conj}\label{c:ls}
Let $G$ be a finite simple group and let $H_p$ be a Sylow $p$-subgroup of $G$. Then there exists an element $x \in G$ such that $H_p \cap H_p^x = 1$ for all $p \in \pi(G)$.
\end{conj}

Since every finite nilpotent group is the direct product of its Sylow subgroups, we immediately observe that Conjecture \ref{c:ls} implies Conjecture \ref{c:nilp}. As noted above, Conjecture \ref{c:ls} is proved in \cite{LS} for all sufficiently large alternating groups. Our main goal in this paper is to prove it for all simple groups of Lie type and all simple sporadic groups, without any conditions on the order of the group.

\begin{theoremm}\label{t:main1}
Conjecture \ref{c:ls} is true for all non-alternating finite simple groups.
\end{theoremm}

By combining this result with the main theorem of \cite{Kurm}, we immediately obtain the following result, which resolves Vdovin's conjecture on nilpotent subgroups of simple groups.

\begin{corr}\label{c:main1}
Conjecture \ref{c:nilp} is true.
\end{corr}

And in view of the connections highlighted in Remark \ref{r:conn}, we get the following corollary, which may be of independent interest.

\begin{corr}\label{c:main2}
Let $G$ be a finite simple group. Then $b(G,H) = 2$ and $d_G(H) = 3$ for every non-trivial nilpotent subgroup $H$ of $G$. 
\end{corr}

\begin{remk}\label{r:main1}
As noted above, Lisi and Sabatini \cite[Theorem 1.4]{LS} prove Conjecture \ref{c:ls} for all sufficiently large alternating groups and their argument relies on recent asymptotic results due to Diaconis et al. \cite{D25} and Eberhard \cite{Eb}. So in view of Theorem \ref{t:main1}, we have a proof of Conjecture \ref{c:ls} for all simple groups, apart from the alternating groups $G = A_n$ with $n \leqs N$ for some unspecified constant $N$. We have verified the  conjecture computationally for all $G = A_n$ with $n \leqs 50$ (see Remark \ref{r:alt0}), but it remains an open problem to prove it for all alternating groups.
\end{remk}

Our proof of Theorem \ref{t:main1} involves a combination of probabilistic and computational methods. In particular, a key feature of the Lisi-Sabatini conjecture is that it allows us to use probabilistic techniques to solve problems about the intersections of nilpotent subgroups, which otherwise would require a detailed case-by-case analysis. This fact was also observed by Eberhard \cite[Remark 4]{Eb}.

In order to outline our approach, let $G$ be a non-alternating finite simple group and let $H_p$ be a Sylow $p$-subgroup of $G$ for each $p \in \pi(G)$.
%
%
Let 
\[
Q(G) = \frac{|\{ x \in G \,:\, \mbox{$H_p \cap H_p^x \ne 1$ for some $p \in \pi(G)$}\}|}{|G|}
\]
be the probability that a uniformly random element $x \in G$ does not satisfy the condition in Conjecture \ref{c:ls}. So our goal is to show that $Q(G)<1$. 

To do this, first observe that
\[
Q(G) \leqs \sum_{p \in \pi(G)} Q_p(G),
\]
where $Q_p(G)$ is the probability that $H_p \cap H_p^x \ne 1$ for a uniformly random $x \in G$. So it suffices to show that
\begin{equation}\label{e:qp}
\sum_{p \in \pi(G)} Q_p(G) < 1,
\end{equation}
which is the same strategy adopted by Lisi and Sabatini in their proof of Conjecture \ref{c:ls_main} for all sufficiently large alternating and symmetric groups. We will show that the inequality in \eqref{e:qp} holds for every non-alternating finite simple group, with the single exception of the classical group ${\rm U}_4(2) \cong {\rm PSp}_4(3)$ (see Theorems \ref{t:ex_main}, \ref{t:class_main} and \ref{t:spor}), noting that the conjecture for ${\rm U}_4(2)$ can be checked directly.

For each $p \in \pi(G)$, we may view $G$ as a transitive permutation group on $\O = G/H_p$, in which case $Q_p(G)$ coincides with the probability that two randomly chosen points in $\O$ do not form a base for $G$. In turn, this implies that
\[
Q_p(G) \leqs \sum_{i=1}^k |x_i^G| \cdot \left(\frac{|x_i^G \cap H_p|}{|x_i^G|}\right)^2 =: \what{Q}_p(G),
\]
where $\{x_1, \ldots, x_k\}$ is a complete set of representatives of the conjugacy classes in $G$ of elements of order $p$ (a more general version of this upper bound was first introduced by Liebeck and Shalev in their proof of \cite[Theorem 1.3]{LSh99}). 

In this way, the problem is essentially reduced to the derivation of an appropriate upper bound on $\what{Q}_p(G)$ for each $p \in \pi(G)$. To do this, we can  appeal to the extensive literature on conjugacy classes of prime order elements in simple groups, which reduces the problem further to estimating $|x_i^G \cap H_p|$. For most classes, we will show that the trivial upper bound $|x_i^G \cap H_p| < |H_p|$ is good enough. But further work is often required when $x_i^G$ is one of the smallest conjugacy classes of elements of order $p$ and in these cases we typically proceed by embedding $H_p$ in a specific subgroup $L$ of $G$, which allows us to work with the upper bound $|x_i^G \cap H_p| \leqs |x_i^G \cap L|$. Here the subgroup $L$ is chosen so that it is easier to count the appropriate number of elements of order $p$ in $L$ than in $H_p$ itself.

In the critical case where $G$ is a group of Lie type in characteristic $r$, the prime divisor $p=r$ requires special attention. Here $N_G(H_r)$ is a Borel subgroup of $G$ with unipotent radical $H_r$ and by considering the opposite Borel subgroup it is easy to see that $H_r \cap H_r^x = 1$ for some $x \in G$. In turn, this implies that 
\[ 
Q_r(G) \leqs 1-\frac{|H_r||N_G(H_r)|}{|G|}
\]
(see Lemma \ref{l:neww}). Then by setting $\pi'(G) = \pi(G) \setminus \{r\}$, it just remains to show that 
\[
\sum_{p \in \pi'(G)} \what{Q}_p(G) < \frac{|H_r||N_G(H_r)|}{|G|}
\]
and this is how we proceed. In particular, this approach allows us to focus entirely on semisimple elements of prime order. 

\begin{remk}\label{r:alt}
Let $G = A_n$ be an alternating group.
Here the main result is \cite[Proposition 4.1]{LS}, which in terms of the above notation gives $Q_2(G) \leqs 0.99$ and $Q_p(G) = O(n^{-1})$ for all large $n$ and all odd primes $p$ (in fact, the main theorem of \cite{Eb} reveals that $Q_2(G) \to 1-\frac{3}{2}e^{-1/2}$ as $n$ tends to infinity; see Theorem \ref{t:deb}(ii) below). Then an easy application of the Prime Number Theorem establishes the upper bound in \eqref{e:qp} for $n \gg 0$ and this completes the proof of Conjecture \ref{c:ls} for all sufficiently large alternating groups. Proving a non-asymptotic version of Conjecture \ref{c:ls} using a probabilistic approach seems to be a difficult problem and we intend to return to this in future work.
\end{remk}

The asymptotic behaviour of the probability $Q_p(G)$ is studied in recent papers by Diaconis et al. \cite{D25} and Eberhard \cite{Eb} in the special case where $G$ is a symmetric or alternating group, with several interesting applications highlighted in \cite{D25}. Here the main results can be summarised as follows (part (i) is \cite[Theorem 1.1(a)]{D25}, while (ii) is \cite[Theorem 1]{Eb}).

\begin{theorem}[Diaconis et al. \cite{D25}; Eberhard \cite{Eb}, 2025]\label{t:deb}
Let $p$ be a prime and let $G = S_n$ or $A_n$.
\begin{itemize}\addtolength{\itemsep}{0.2\baselineskip}
\item[{\rm (i)}] If $p>2$, then $Q_p(G) \to 0$ as $n \to \infty$.
\item[{\rm (ii)}] If $p = 2$, then $Q_p(G) \to 1-\a e^{-1/2}$ as $n \to \infty$, where $\a = 1$ if $G = S_n$ and $\a = \frac{3}{2}$ if $G = A_n$.
\end{itemize}
\end{theorem}

Our probabilistic proof of Theorem \ref{t:main1} allows us to establish a strong version of Theorem \ref{t:deb} for groups of Lie type (see the asymptotic statements in Theorems \ref{t:ex_main} and \ref{t:class_main}).

\begin{theoremm}\label{t:main_asymptotic}
Let $G$ be a finite simple group of Lie type over a field $\mathbb{F}_q$ of  characteristic $r$ and set 
\[
\b(G) = \sum_{p \in \pi(G) \setminus \{r\}} Q_p(G), \;\; \gamma(G) = \sum_{p \in \pi(G) \setminus\{2,r\}} Q_p(G).
\]
Then $\gamma(G) \to 0$ as $|G| \to \infty$. In addition, if $G \ne {\rm L}_2(q)$ then $\b(G) \to 0$ as $|G| \to \infty$.
\end{theoremm}

By combining this with Theorem \ref{t:deb}, we obtain the following corollary. 

\begin{corr}\label{c:main_asymptotic}
Let $(G_i,p_i)$ be a sequence of pairs, where $G_i$ is a finite simple group, $p_i$ is a prime and $|G_i| \to \infty$. Then either 
\[
\mbox{$Q_{p_i}(G_i)\to 0$ as $i \to \infty$,}
\]
or there exists a subsequence of pairs $(G_{i_k},p_{i_k})$ such that one of the following holds for all $k$:
\begin{itemize}\addtolength{\itemsep}{0.2\baselineskip}
\item[{\rm (a)}] $G_{i_k} = A_{n_k}$ and $p_{i_k} = 2$;
\item[{\rm (b)}] $G_{i_k} = \mathrm{L}_2(q_k)$, $q_k$ is odd and $p_{i_k} = 2$; 
\item[{\rm (c)}] $G_{i_k}$ is a group of Lie type defined over a field of characteristic $p_{i_k}$.
\end{itemize}
\end{corr}

\begin{remk}\label{r:main3}
Let us briefly comment on the special subsequences arising in the statement of Corollary \ref{c:main_asymptotic}. For (a), we refer to Theorem \ref{t:deb}(ii). And for (b), observe that if $G = {\rm L}_2(q)$ and $q$ is a Mersenne prime, then 
\[
Q_2(G) = 1 - \frac{s|H_2|}{|G:H_2|},
\]
where $s$ is the number of regular orbits of $H_2 = D_{q+1}$ on $G/H_2$. By \cite[Lemma 7.9]{BH20} we have $s = (q-3)/4$ and thus
\[
Q_2(G) = \frac{1}{2} + \frac{q+3}{2q(q-1)}.
\]
In particular, if there are infinitely many Mersenne primes (as expected), then there exists a sequence of groups $G_i = {\rm L}_2(q_i)$ such that $|G_i| \to \infty$ and $Q_2(G_i) \to \frac{1}{2}$ as $i$ tends to infinity. Finally, in case (c) we refer to \cite[Lemma 3.13]{Z_92p}, which implies that if $G$ is a simple group of Lie type in characteristic $r$, and $G \ne {\rm Sp}_4(2)'$, $G_2(2)'$, ${}^2G_2(3)'$ or ${}^2F_4(2)'$, then
\[
Q_{r}(G) = 1-\frac{|H_r||N_G(H_r)|}{|G|}.
\]
So for example, if $G = {\rm L}_n(q)$ and $q = r^f$ then $|N_G(H_r):H_r| = \frac{1}{d}(q-1)^{n-1}$ with $d = (n,q-1)$ and thus
\[
Q_r(G) = 1 - \frac{q^{n(n-1)/2}(q-1)^{n-1}}{\prod_{i=2}^n(q^i-1)} \to 1
\]
if $q$ is fixed and $n$ tends to infinity.
\end{remk}

As remarked above, our proof of Theorem \ref{t:main1} allows us to show that \eqref{e:qp} holds for every non-alternating finite simple group, with the single exception of ${\rm U}_4(2) \cong {\rm PSp}_4(3)$. Putting the special case ${\rm U}_4(2)$ to one side, this immediately settles Conjecture \ref{c:ls}, and hence Conjecture \ref{c:nilp}, for all non-alternating simple groups. In addition, it allows us to establish the following generalisation of Conjecture \ref{c:nilp} (as before, the conclusion can be checked directly for ${\rm U}_4(2)$). 

\begin{theoremm}\label{t:main2}
Let $G$ be a non-alternating finite simple group and let $A$ and $B$ be nilpotent subgroups of $G$. Then $A \cap B^x = 1$ for some $x \in G$. 
\end{theoremm}

To see this, let $A$ and $B$ be nilpotent subgroups of $G$ and for each $p \in \pi(G)$ let $A_p$ and $B_p$ be the unique Sylow $p$-subgroups of $A$ and $B$, respectively. Setting
\[
R_p(A,B) = \frac{|\{x \in G \,:\, A_p \cap B_p^x \ne 1\}|}{|G|}
\]
we observe that $A \cap B^x = 1$ for some $x \in G$ if $\sum_p R_p(A,B) < 1$. Now if we embed $A_p \leqs H_p$ and $B_p \leqs K_p$, where $H_p$ and $K_p$ are Sylow $p$-subgroups of $G$, then 
\[
R_p(A,B) \leqs \frac{|\{x \in G \,:\, H_p \cap K_p^x \ne 1\}|}{|G|} = Q_p(G)
\]
and so the inequality in \eqref{e:qp} yields $A \cap B^x = 1$ for some $x \in G$.

For sporadic groups, Theorem \ref{t:main2} gives a new proof of a result of Zenkov \cite{Z_spor}. Moreover, by combining Theorem \ref{t:main2} with earlier work of Zenkov on alternating groups \cite{Z_sym}, we obtain the following corollary which can be viewed as a natural generalisation of Theorem \ref{t:mz} and Corollary \ref{c:main1}. In the language of \cite{AB}, this shows that every pair of nilpotent subgroups of a simple group is regular.

\begin{corr}\label{c:main3}
Let $G$ be a finite simple group and let $A$ and $B$ be nilpotent subgroups of $G$. Then $A \cap B^x = 1$ for some $x \in G$. 
\end{corr}

Finally, let us observe that our proof of Theorem \ref{t:main1} is independent of Theorem \ref{t:mz}. In particular, we obtain a new proof of Theorem \ref{t:mz} for simple groups of Lie type, which does not rely on the earlier work in \cite{Green, Michler, Willems} concerning $p$-blocks and defect groups. Moreover, our proof is quantitative and it allows us to draw conclusions concerning the proportion of elements $x \in G$ with $H \cap H^x = 1$, which is not possible via the original existence proof in \cite{ZM}. 


\vs

\noindent \textbf{Notation.} Let $G$ be a finite group and let $n$ be a positive integer. We will write $C_n$, or just $n$, for a cyclic group of order $n$ and $G^n$ for the direct product of $n$ copies of $G$. An unspecified extension of $G$ by a group $H$ will be denoted by $G.H$; if the extension splits then we may write $G{:}H$. We will use $i_n(G)$ for the number of elements of order $n$ in $G$, and we will sometimes write $[n]$ for an unspecified soluble group of order $n$. Throughout the paper, we adopt the standard notation for simple groups of Lie type from \cite{KL}. In addition, we write $(a,b)$ for the highest common factor of the positive integers $a$ and $b$. And for a prime $p$, we use $(n)_p$ to denote the $p$-part of $n$. All logarithms in this paper are base $2$.

\vs

\noindent \textbf{Acknowledgements.} We thank an anonymous referee for their helpful comments and suggestions on an earlier version of the paper. The second author thanks the London Mathematical Society for their support as an LMS Early Career Research Fellow at the University of St Andrews.

\section{Preliminaries}\label{s:prel}

In this section we present a collection of preliminary results, which will be needed in the proofs of our main theorems.

\subsection{Number theory}\label{ss:nt}

We will require one or two basic number-theoretic results. Given a positive integer $m$ and a prime number $p$, we write $(m)_p$ for the largest power of $p$ dividing $m$ (so for example, $(18)_3 = 9$).

\begin{lem}\label{l:nt}
Let $t$ be a prime power, let $d$ be a positive integer and let $p$ be a prime divisor of $t-\e$, where $\e=\pm 1$.
\begin{itemize}\addtolength{\itemsep}{0.2\baselineskip}
\item[{\rm (i)}] If $p=2$, then
\[
(t^d-\e)_p = \left\{\begin{array}{ll}
(t-\e)_p & \mbox{if $d$ is odd} \\
(t^2-1)_p(d/2)_p & \mbox{if $d$ is even and $\e=1$} \\
2 & \mbox{if $d$ is even and $\e=-1$.}
\end{array}\right.
\]
\item[{\rm (ii)}] If $p$ is odd, then
\begin{align*}
(t^d-\e)_p & = \left\{ \begin{array}{ll}
1 & \mbox{if $d$ is even and $\e=-1$} \\
(t-\e)_p(d)_p & \mbox{otherwise}
\end{array}\right. \\
(t^d+\e)_p & = \left\{ \begin{array}{ll}
(t-\e)_p(d)_p & \mbox{if $d$ is even and $\e=-1$} \\
1 & \mbox{otherwise.}
\end{array}\right.
\end{align*}
\end{itemize}
\end{lem}

\begin{proof}
This is \cite[Lemma 2.1(i)]{BBGT}.
\end{proof}

We will also need the following results concerning factorials and odd prime divisors.

\begin{lem}\label{l:factorial}
Let $m$ be a positive integer and let $p$ be a prime. Then $p^m(m!)_p = ((pm)!)_p$ and $(m!)_p < p^k$, where $k = m/(p-1)$.
\end{lem}

\begin{proof}
The first claim is clear. For the second statement, set $(m!)_p = p^{\ell}$ and note that 
\[
\ell = \sum_{i=1}^{\infty}\left\lfloor \frac{m}{p^i} \right\rfloor < m \sum_{i=1}^{\infty}p^{-i} = \frac{m}{p-1}. \qedhere
\]
\end{proof}

\begin{lem}\label{l:primes}
Let $d$ be a positive integer. Then $d+1$ has at most $\log d$ distinct odd prime divisors.
\end{lem}

\begin{proof}
We may assume $d+1$ is divisible by an odd prime, so we can write 
\[
d+1 = 2^{a}p_1^{a_1} \cdots p_k^{a_k},
\] 
where the $p_i$ are distinct odd primes and we have $a \geqs 0$ and $a_i \geqs 1$ for all $i$. Setting $b = \sum_ia_i$, we see that $d \geqs 2^{a+b}$ and thus $k \leqs b \leqs \log d$ as required.
\end{proof}

\subsection{Probability}\label{ss:prob}

Let $G$ be a finite simple group and let $\pi(G)$ be the set of prime divisors of $|G|$. Fix a prime $p \in \pi(G)$ and let $H_p$ be a Sylow $p$-subgroup of $G$. In view of Theorem \ref{t:main1}, we are interested in the following property:
\begin{equation}\label{e:star}
\mbox{\emph{There exists an element $x \in G$ such that $H_p \cap H_p^x = 1$ for all $p \in \pi(G)$}}
\tag{$\star$}
\end{equation}
and our main goal is to show that \eqref{e:star} holds for all non-alternating simple groups.

Let
\[
Q_p(G) = \frac{|\{ x \in G \,:\, H_p \cap H_p^x \ne 1\}|}{|G|}
\]
be the probability that $H_p \cap H_p^x \ne 1$ for a uniformly random element $x \in G$ and note that $Q_p(G) < 1$ by Theorem \ref{t:mz}. The following observation is immediate.

\begin{lem}\label{l:prob}
Property \eqref{e:star} holds if $\displaystyle \sum_{p \in \pi(G)} Q_p(G) < 1$.
\end{lem}

We will also make use of the following result.

\begin{lem}\label{l:neww}
Let $G$ be a finite group with $Q_p(G)<1$ and let $H_p$ be a Sylow $p$-subgroup of $G$. Then
\begin{equation}\label{e:free}
Q_p(G) \leqs 1-\frac{|H_p||N_G(H_p)|}{|G|}.
\end{equation}
\end{lem}

\begin{proof}
Set $K = N_G(H_p)$ and observe that $H_p$ is the unique Sylow $p$-subgroup of $K$. Since $Q_p(G)<1$, we may fix an element $x \in G$ such that $H_p \cap H_p^x = 1$. And since $H_p^x \cap K$ is a $p$-subgroup of $K$, it follows that $H_p^x \cap K \leqs H_p$ 
and thus $H_p^x \cap K = H_p^x \cap H_p = 1$.

Now consider the double coset $KxH_p$, which has size $|K||H_p|$ since $H_p^x \cap K = 1$. For any element $z = kxh$ in this double coset, we have
\[
H_p \cap H_p^z = H_p \cap H_p^{kxh} = H_p \cap H_p^{xh} = (H_p^{h^{-1}} \cap H_p^x)^h = (H_p \cap H_p^x)^h = 1.
\]
This allows us to conclude that there are at least $|K||H_p|$ elements $x \in G$ with $H_p \cap H_p^x = 1$ and the result follows.
\end{proof}


For a given prime $p \in \pi(G)$, let us observe that $Q_p(G)$ coincides with the probability that a random pair of cosets in $G/H_p$ do not form a base for $G$. Then by 
arguing as in the proof of \cite[Theorem 1.3]{LSh99}, it follows that 
\[
Q_p(G) \leqs \widehat{Q}_p(G) := \sum_{i=1}^k \frac{|x_i^G \cap H_p|^2}{|x_i^G|},
\]
where $\{x_1, \ldots, x_k\}$ is a complete set of representatives of the conjugacy classes in $G$ of elements of order $p$. In particular, \eqref{e:star} holds if
\begin{equation}\label{e:hati}
\sum_{p \in \pi(G)} \widehat{Q}_p(G) < 1.
\end{equation}

\begin{rem}
Let us record several observations on the above set-up.

\begin{itemize}\addtolength{\itemsep}{0.2\baselineskip}
\item[{\rm (a)}] We have
\[
Q_p(G) = 1 - \frac{s|H_p|}{|G:H_p|},
\]
where $s$ is the number of regular $H_p$-orbits on $G/H_p$. Equivalently, $s$ is the number of $(H_p,H_p)$ double cosets in $G$ of size $|H_p|^2$.

\item[{\rm (b)}] Let $G$ be a finite simple group of Lie type defined over a field $\mathbb{F}_q$ of characteristic $r$. If we exclude the special cases ${\rm Sp}_4(2)'$, $G_2(2)'$, ${}^2G_2(3)'$ and ${}^2F_4(2)'$, then a result of Zenkov \cite[Lemma 3.13]{Z_92p} shows that 
\[
Q_r(G) = 1-\frac{|H_r||N_G(H_r)|}{|G|},
\]
so the upper bound in \eqref{e:free} is in fact an equality in this situation.

In addition, for each prime divisor $p \ne r$ of $|G|$ we will derive an explicit upper bound of the form 
\[
\what{Q}_p(G) \leqs f_p(t,q),
\]
where $t$ is the rank of $G$. Then by combining these estimates, we get
\[
\sum_{p \in \pi(G) \setminus \{r\}} \what{Q}_p(G) \leqs f(t,q),
\]
which we can use with the upper bound in \eqref{e:free} (with $p=r$) in order to verify the inequality in \eqref{e:hati}, with the possible exception of a handful of special cases that can be treated directly. This is our basic approach to the proof of Theorem \ref{t:main1} for groups of Lie type. In addition, we will see that $f(t,q) \to 0$ as $t$ or $q$ tends to infinity (unless $G = {\rm L}_2(q)$ and $q$ is odd) and this is how we will prove Theorem \ref{t:main_asymptotic}. 

\item[{\rm (c)}] As noted in Section \ref{s:intro}, the key inequality $\sum_p Q_p(G)<1$ in Lemma \ref{l:prob} implies that if $A$ and $B$ are nilpotent subgroups of $G$, then $A \cap B^x = 1$ for some $x \in G$. So a detailed analysis of $\sum_p Q_p(G)$ for non-alternating simple groups plays an essential role in our proof of Theorem \ref{t:main2}.
\end{itemize}
\end{rem}

The following elementary result will be useful for computing an upper bound on $\widehat{Q}_p(G)$.

\begin{lem}\label{l:est}
Suppose $x_1, \ldots, x_m$ represent distinct $G$-classes such that $\sum_i |x_i^G \cap H_p| \leqs a$ and $|x_i^G| \geqs b$ for all $i$. Then
\[
\sum_{i=1}^m \frac{|x_i^G \cap H_p|^2}{|x_i^G|} \leqs a^2b^{-1}.
\]
\end{lem}

\begin{proof}
This is \cite[Lemma 2.1]{B07}.
\end{proof}

\section{Exceptional groups of Lie type}\label{s:excep}

In this section we will prove Theorems \ref{t:main1}, \ref{t:main_asymptotic} and \ref{t:main2} for all simple exceptional groups of Lie type; the classical groups will be handled in Section \ref{s:class} and we conclude by dealing with the sporadic groups in Section \ref{s:spor}. We begin by setting up some general notation that we will use for all simple groups of Lie type. 

Let $G$ be a finite simple group of Lie type over $\mathbb{F}_q$, where $q = r^f$ for some prime $r$ and integer $f$. Let $\pi(G)$ be the set of prime divisors of $|G|$ and set 
\[
\pi'(G) = \pi(G) \setminus \{r\}.
\]
Let $H_r$ be a Sylow $r$-subgroup of $G$, which means that $B = N_G(H_r)$ is a Borel subgroup of $G$ with unipotent radical $H_r$. By choosing $x \in G$ so that $B^x$ is the opposite Borel subgroup, we deduce that $H_r \cap H_r^x = 1$ (without appealing to Theorem \ref{t:mz}). So in view of \eqref{e:free} and Lemma \ref{l:prob}, we observe that \eqref{e:star} holds (see Section \ref{ss:prob}) if $\Sigma(G) < 1$, where
\begin{equation}\label{e:sigma}
\Sigma(G) = 1 - \frac{|H_r||N_G(H_r)|}{|G|} + \sum_{p \in \pi'(G)} \what{Q}_p(G).
\end{equation}

The main result of this section is the following, which immediately yields Theorems \ref{t:main1}, \ref{t:main_asymptotic} and \ref{t:main2} for exceptional groups.

\begin{thm}\label{t:ex_main}
Let $G$ be a finite simple exceptional group of Lie type and set
\[
\a(G) = \sum_{p \in \pi(G)} Q_p(G),\;\; \b(G) = \sum_{p \in \pi'(G)} Q_p(G).
\]
Then $\a(G)<1$. In addition, $\b(G) \to 0$ as $|G| \to \infty$.
\end{thm}

It is convenient to handle certain low rank groups defined over small fields by direct computation in {\sc Magma} \cite{magma} (working with version V2.28-21). This is recorded in the following result.

\begin{lem}\label{l:ex_comp}
The conclusion to Theorem \ref{t:ex_main} holds when $G$ is one of the following:
\[
G_2(2)',\; G_2(3), \; G_2(4), \; {}^2G_2(3)', \; {}^2F_4(2)'.
\]
\end{lem}

\begin{proof}
Let $r$ be the defining characteristic. 
We begin by using the function 
\[
\mbox{\texttt{AutomorphismGroupSimpleGroup}}
\]
to construct $G$ as a permutation group of degree $n$ (it is helpful to observe that $G_2(2)' \cong {\rm U}_3(3)$ and ${}^2G_2(3)' \cong {\rm L}_2(8)$, so we can work directly with the classical groups ${\rm U}_3(3)$ and ${\rm L}_2(8)$ in these two special cases). Then for each prime $p \in \pi(G)$ we construct a Sylow $p$-subgroup $H_p$ of $G$. In addition, if $p \ne r$ then we construct a set of representatives $\mathcal{C} = \{y_1, \ldots, y_s\}$ of the conjugacy classes in $H_p$ of elements of order $p$, which we then partition $\mathcal{C} = \mathcal{C}_1 \cup \cdots \cup \mathcal{C}_{\ell}$ so that $y_i$ and $y_j$ are in the same subset if and only if they are conjugate in $G$. Write $\mathcal{C}_i = \{ y_{i,1}, \ldots, y_{i,k_i}\}$ for $i = 1, \ldots, \ell$. Then
\[
\widehat{Q}_p(G) = \sum_{i=1}^{\ell} |y_{i,1}^G|^{-1}\left(\sum_{j=1}^{k_j}|y_{i,j}^{H_p}|\right)^2
\]
and in each case we verify the bound $\Sigma(G) < 1$. For example, if $G = {}^2F_4(2)'$ then $n = 1755$ and we compute
\[
\Sigma(G) = \frac{346991}{449280},
\]
with the computation taking roughly $1$ second and 30MB of memory.
\end{proof}

For handling the remaining cases, the following result from \cite{BTh} will be a useful tool.

\begin{lem}\label{l:bounds}
Let $G \ne G_2(2)', {}^2F_4(2)', {}^2G_2(3)'$ be a simple exceptional group of Lie type over $\mathbb{F}_q$. Let $x_2 \in G$ be a semisimple element of odd prime order. In addition, if $q$ is odd then let $x_1 \in G$ be an involution. Then $|x_i^G| > \ell_i$, where $\ell_1$ and $\ell_2$ are given in Table \ref{tab:classes}.
\end{lem}

\begin{proof}
This is a special case of \cite[Proposition 2.11]{BTh}.
\end{proof}

{\small
\begin{table}
\[
\begin{array}{lllllll} \hline
G & \ell_1 & \ell_2 & & & &  \\ \hline
E_8(q) & q^{112} & (q-1)q^{113} & \mbox{\hspace{5mm}  } & G_2(q),\, q \geqs 3 & q^8 & (q-1)q^5 \\
E_7(q) & \frac{1}{2}(q-1)q^{53} & (q-1)q^{53} & & {}^3D_4(q) & q^{16} & (q-1)q^{17} \\
E_6(q) & q^{32} & q^{32} & & {}^2F_4(q),\, q \geqs 8 & \mbox{--} & (q-1)q^{17} \\
{}^2E_6(q) & (q-1)q^{31} & (q-1)q^{31} & & {}^2G_2(q),\, q \geqs 27 & (q-1)q^3 & \frac{1}{2}q^6 \\
F_4(q) & q^{16} & (q-1)q^{29} & & {}^2B_2(q) & \mbox{--} & \frac{1}{2}q^4 \\ \hline
\end{array}
\]
\caption{The lower bounds $|x_i^G|> \ell_i$ in Lemma \ref{l:bounds}}
\label{tab:classes}
\end{table}
}

We are now ready to prove Theorem \ref{t:ex_main}.

\begin{proof}[Proof of Theorem \ref{t:ex_main}]
Let $G$ be a finite simple exceptional group of Lie type over $\mathbb{F}_q$, where $q = r^f$ is a power of the prime $r$. Since a very similar argument applies in each case, we will only give details for the groups ${}^2E_6(q)$, $F_4(q)$ and $G_2(q)'$. It will be convenient to define
\[
\what{\b}(G) = \sum_{p \in \pi'(G)} \what{Q}_p(G). 
\]

\vs

\noindent \emph{Case 1. $G = {}^2E_6(q)$}

\vs

First assume $G = {}^2E_6(q)$ and note that  
\[
|G| = \frac{1}{d}q^{36}(q^2-1)(q^5+1)(q^6-1)(q^8-1)(q^9+1)(q^{12}-1),
\]
where $d = (3,q+1)$. In addition, we have $|N_G(H_r):H_r| = \frac{1}{d}(q^2-1)^2(q-1)^2$ and 
\[
\frac{|H_r||N_G(H_r)|}{|G|} = \frac{q^{72}(q^2-1)^2(q-1)^2}{d|G|} > \left(1-q^{-1}\right)^6,
\]
whence $\Sigma(G)<1$ if 
\begin{equation}\label{e:2e6_prob}
\what{\b}(G) \leqs \left(1-q^{-1}\right)^6.
\end{equation}

Fix a prime $p \in \pi'(G)$. For $p=2$, Lemma \ref{l:nt}(i) gives   
\[
|H_2| = 2^3((q^2-1)_2)^4((q+1)_2)^2 \leqs 2^7(q+1)^6 = a_1
\]
and Lemma \ref{l:bounds} states that $|x^G|>(q-1)q^{31} = b_1$ for every involution $x \in G$. So by appealing to Lemma \ref{l:est} we deduce that 
$\what{Q}_2(G) < a_1^2b_1^{-1}$, where $a_1$ and $b_1$ are defined by the equations they appear in above.
 
Now assume $p$ is odd and let $m$ be minimal such that $p$ divides $q^m-1$, in which case $m \in \{1,2,3,4,6,8,10,12,18\}$. By carefully considering each possibility for $m$ in turn, it is straightforward to show that $|H_p| \leqs 3^3(q+1)^6 = a_2$. For example, if $m = 6$ then $p \geqs 7$ and
\[
|H_p| = (q^6-1)_p(q^9+1)_p(q^{12}-1)_p = ((q^2-q+1)_p)^3.
\]
Since $|x^G|>(q-1)q^{31} = b_2$ for all $x \in G$ of order $p$ (see Lemma \ref{l:bounds}), Lemma \ref{l:est} implies that $\what{Q}_p(G) < a_2^2b_2^{-1}$.
Moreover, since each $p \in \pi'(G)$ divides $n = (q^4+1)(q^5+1)(q^9+1)(q^{12}-1)$, it follows that $|\pi'(G)| < \log n < 31\log q$ and thus  
\[
\what{\b}(G) < (1-\delta_{2,r})a_1^2b_1^{-1} + 31\log q \cdot a_2^2b_2^{-1},
\]
where $\delta_{2,r}$ is the usual Kronecker delta (so $\delta_{2,r} = 1$ if $r=2$, otherwise $\delta_{2,r} = 0$). 

This upper bound immediately implies that $\b(G) \to 0$ as $q \to \infty$, so the asymptotic statement in Theorem \ref{t:ex_main} holds for $G = {}^2E_6(q)$. In addition, it is routine to check that the inequality in \eqref{e:2e6_prob} is satisfied for all $q \geqs 3$. 

Finally, suppose $q=2$ and $p \in \pi'(G)$. Here $|\pi'(G)| = 7$ and $|H_p| \leqs 3^9 = a$. In addition, by inspecting the character table of $G$ (see \cite{GAPCTL}) we observe that $|x^G| \geqs 1319933815200=b$ for all $x \in G$ of order $p$. Therefore, $\what{\b}(G) \leqs 7a^2b^{-1} < 2^{-6}$
and the result follows.

\vs

\noindent \emph{Case 2. $G = F_4(q)$}

\vs

Now suppose $G = F_4(q)$. Here 
\[
|G| = q^{24}(q^2-1)(q^6-1)(q^8-1)(q^{12}-1) 
\]
and $|N_G(H_r):H_r| = (q-1)^4$, whence $\Sigma(G)<1$ if  
\begin{equation}\label{e:f4_prob}
\what{\b}(G) \leqs \left(1-q^{-1}\right)^4.
\end{equation}

Fix a prime $p \in \pi'(G)$. For $p=2$ we compute  
\[
|H_2| = 2^3((q^2-1)_2)^4 \leqs 2^7(q+1)^4 = a_1
\]
and we note that $|x^G|>q^{16} = b_1$ for every involution $x \in G$ (see Lemma \ref{l:bounds}). This gives $\what{Q}_2(G) < a_1^2b_1^{-1}$. On the other hand, if $p$ is odd then $|H_p| \leqs 3^2(q+1)^4 = a_2$ and  Lemma \ref{l:bounds} gives $|x^G|>(q-1)q^{29} = b_2$ for all $x \in G$ of order $p$, whence $\what{Q}_p(G) < a_2^2b_2^{-1}$. In addition, since $|\pi'(G)| \leqs 20\log q$, it follows that
\[
\what{\b}(G) < (1-\delta_{2,r})a_1^2b_1^{-1} + 20\log q \cdot a_2^2b_2^{-1}.
\]
Once again, it is clear from this bound that we have $\b(G) \to 0$ as $q \to \infty$. One can also check that the inequality in \eqref{e:f4_prob} is satisfied for all $q \ne 3$. 

Finally, suppose $q=3$. Here $|\pi'(G)| = 6$ and by embedding $H_2$ in a maximal subgroup $L = 2.\O_9(3) = {\rm Spin}_9(3)$, we can use {\sc Magma} to show that $i_2(H_2) = 1615$. The result now follows since 
\[
\what{\b}(G) \leqs 1615^2b_1^{-1} + 5a_2^2b_2^{-1} < \left(\frac{2}{3}\right)^4,
\]
where $b_1$, $a_2$ and $b_2$ are defined as above. 

\vs

\noindent \emph{Case 3. $G = G_2(q)'$}

\vs

Finally, suppose $G = G_2(q)'$. In view of Lemma \ref{l:ex_comp}, we may assume $q \geqs 5$. Since we have $|G| = q^{6}(q^2-1)(q^6-1)$ and $|N_G(H_r)| = q^6(q-1)^2$, it follows that $\Sigma(G)<1$ if  
\begin{equation}\label{e:g2_prob}
\what{\b}(G) \leqs \left(1-q^{-1}\right)^2.
\end{equation}

Fix a prime $p \in \pi'(G)$. For $p=2$ we have $|H_2| = ((q^2-1)_2)^2 \leqs 2^2(q+1)^2 = a_1$ and $|x^G| = q^4(q^4+q^2+1) = b_1$ for every involution $x \in G$, which yields $\what{Q}_2(G) \leqs a_1^2b_1^{-1}$. And for $p \geqs 3$ we find that $|H_p| \leqs 3(q+1)^2 = a_2$, while Lemma \ref{l:bounds} shows that $|x^G|>(q-1)q^{5} = b_2$ for all $x \in G$ of order $p$. Therefore, $\what{Q}_p(G) < a_2^2b_2^{-1}$ and since $|\pi'(G)| \leqs 6\log q$ we deduce that
\[
\what{\b}(G) < (1-\delta_{2,r})a_1^2b_1^{-1} + 6\log q \cdot a_2^2b_2^{-1} < \left(1-q^{-1}\right)^2
\]
for all $q \geqs 19$. In addition, we see that $\b(G) \to 0$ as $q \to \infty$, so the desired asymptotic statement holds in this case.

Finally, suppose $5 \leqs q \leqs 17$. Here it is easy to improve the above estimates in order to verify the bound in \eqref{e:g2_prob}. For example, suppose $q = 5$. Here $|H_2| = 2^6 = a_1$ and $|H_p| \leqs 31 = a_2$ for all odd primes $p \in \pi'(G) = \{2,3,7,31\}$. Therefore,
\[
\what{\b}(G) \leqs a_1^2b_1^{-1} + 3a_2^2b_2^{-1} = \frac{1958753}{8137500} < \left(\frac{4}{5}\right)^2
\]
as required, where $b_1$ and $b_2$ are defined as above.
\end{proof}

\section{Classical groups}\label{s:class}

In this section we prove Theorems \ref{t:main1} and \ref{t:main2} for classical groups. In addition, we complete the proof of Theorem \ref{t:main_asymptotic}.

Let $G$ be a finite simple classical group over $\mathbb{F}_q$, where $q = r^f$ for a prime $r$, and let $V$ be the natural module for $G$. Due to the existence of exceptional isomorphisms between some of the low-dimensional classical groups (see \cite[Proposition 2.9.1]{KL}), we may assume $G$ is one of the following:
\[
\begin{array}{ll}
{\rm L}_n(q) & \mbox{$n \geqs 2$, with $q \geqs 5$ if $n=2$} \\
{\rm U}_n(q) & \mbox{$n \geqs 3$} \\ 
{\rm PSp}_n(q) & \mbox{$n \geqs 4$ even, $(n,q) \ne (4,2)$} \\
\O_n(q) & \mbox{$n \geqs 7$ odd, $q$ odd} \\
{\rm P\O}_n^{\pm}(q) & \mbox{$n \geqs 8$ even.} 
\end{array}
\]

\renewcommand{\arraystretch}{1.4}

As before, set $\pi'(G) = \pi(G) \setminus \{r\}$ and define
\begin{equation}\label{e:defs}
\begin{array}{llll}
\displaystyle \a(G) = \sum_{p \in \pi(G)} Q_p(G) & & & \displaystyle \b(G) = \sum_{p \in \pi'(G)} Q_p(G) \\
\displaystyle \gamma(G) = \sum_{p \in \pi'(G) \setminus \{2\}} Q_p(G) & & & 
\displaystyle \what{\b}(G) = \sum_{p \in \pi'(G)} \what{Q}_p(G)
\end{array}
\end{equation}
In addition, define $\Sigma(G)$ as in \eqref{e:sigma} and recall that $\a(G)<1$ if $\Sigma(G)<1$. Our main result for classical groups is the following.

\renewcommand{\arraystretch}{1}

\begin{thm}\label{t:class_main}
Let $G$ be a finite simple classical group over $\mathbb{F}_q$.

\begin{itemize}\addtolength{\itemsep}{0.2\baselineskip}
\item[{\rm (i)}] If $G \ne {\rm U}_4(2)$, then $\a(G)<1$.
\item[{\rm (ii)}] We have $\gamma(G) \to 0$ as $|G| \to \infty$. 
\item[{\rm (iii)}] If $G \ne {\rm L}_2(q)$, then $\b(G) \to 0$ as $|G| \to \infty$.
\end{itemize}
\end{thm}

\begin{rem}\label{r:class}
Let us record a couple of comments on the statement of Theorem \ref{t:class_main}.

\begin{itemize}\addtolength{\itemsep}{0.2\baselineskip}
\item[{\rm (a)}] First note that the group $G = {\rm U}_4(2) \cong {\rm PSp}_4(3)$ in part (i) is a genuine exception. Indeed, we have $\pi(G) = \{2,3,5\}$ and using {\sc Magma} we compute
\[
Q_2(G) = \frac{71}{135},\; Q_3(G) = \frac{79}{160},\; Q_5(G) = \frac{1}{1296}, 
\]
which means that $\a(G) = 2645/2592 > 1$. 

\item[{\rm (b)}] Suppose $G = {\rm L}_2(q)$ with $q$ odd. Here the proof of Lemma \ref{l:l2} shows that $\gamma(G) \to 0$ as $q \to \infty$.
However, the asymptotic behaviour of $Q_2(G)$ is rather different and we are not able to conclude that $Q_2(G) \to 0$ as $q \to \infty$. For example, as discussed in Remark \ref{r:main3}, we have
\[
Q_2(G) = \frac{1}{2}+\frac{q+3}{2q(q-1)}
\]
if $q$ is a Mersenne prime. This explains the condition $G \ne {\rm L}_2(q)$ in part (iii).
\end{itemize}
\end{rem}

\begin{rem}\label{r:u42}
Let us briefly explain how we establish Theorems \ref{t:main1} and \ref{t:main2} in the special case $G = {\rm U}_4(2)$, where we have $\pi(G) = \{2,3,5\}$ and $\a(G)>1$. 

\begin{itemize}\addtolength{\itemsep}{0.2\baselineskip}
\item[{\rm (a)}] In order to verify Conjecture \ref{c:ls}, we first use {\sc Magma} to construct $G$ as a permutation group of degree $45$. We fix a Sylow $5$-subgroup $H_5$ of $G$ and we construct the complete set $\Gamma_p$ of Sylow $p$-subgroups of $G$ for $p \in \{2,3\}$. Then for each pair $(H_2,H_3) \in \Gamma_2 \times \Gamma_3$ we use random search to find an element $x \in G$ such that $H_p \cap H_p^x = 1$ for all $p \in \{2,3,5\}$. This is an entirely straightforward computation and it establishes Conjecture \ref{c:ls} for $G = {\rm U}_4(2)$.

\item[{\rm (b)}] Similarly, to check Theorem \ref{t:main2}, we first use the function \texttt{NilpotentSubgroups} to construct a complete set of representatives of the conjugacy classes of nilpotent subgroups of $G$ (there are $48$ such classes). Then for each pair $(A,B)$ of representatives we use random search to identify an element $x \in G$ with $A \cap B^x = 1$. Once again, this is an easy calculation.
\end{itemize}
\end{rem}

For a positive integer $m$ we set
\begin{equation}\label{e:pm}
\mathcal{P}_m = \{ p \in \pi'(G) \,:\, \mbox{$p \ne 2$ is a primitive prime divisor of $q^m-1$}\},
\end{equation}
which means that each $p \in \mathcal{P}_m$ divides $q^m-1$, but it does not divide $q^i-1$ for any $1 \leqs i < m$. Since every prime divisor $p \ne r$ of $|G|$ is a primitive prime divisor of $q^m-1$ for some $m$, it follows that 
\[
\what{\b}(G) = (1-\delta_{2,r})\what{Q}_2(G) + \sum_{m \geqs 1}\sum_{p \in \mathcal{P}_m} \what{Q}_p(G)
\]
and we define
\begin{equation}\label{e:alpha}
\rho_0 = \what{Q}_2(G),\;\; \rho_1 = \sum_{p \in \mathcal{P}_1} \what{Q}_p(G), \;\; \rho_2 = \sum_{p \in \mathcal{P}_2} \what{Q}_p(G), \;\; \rho_3 = \sum_{m \geqs 3}\sum_{p \in \mathcal{P}_m} \what{Q}_p(G).
\end{equation}
So in terms of this notation, we have 
\[
\what{\b}(G) = (1-\delta_{2,r})\rho_0 + \rho_1 + \rho_2 + \rho_3.
\]

For each finite simple classical group $G$ over $\mathbb{F}_q$, we will produce an explicit upper bound of the form $\rho_i \leqs f_i(n,q)$ for each $i \in \{0,1,2,3\}$, where $n$ is the dimension of the natural module for $G$. By combining these estimates, we will show that
\[
\sum_{i=0}^3 f_i(n,q) < \frac{|H_r||N_G(H_r)|}{|G|}
\]
with the possible exception of a handful of special cases, which allows us to conclude that $\Sigma(G)<1$ and thus $\a(G)<1$, as in Theorem \ref{t:class_main}(i). In addition, it will be clear in every case that we have $f_i(n,q) \to 0$ as $n$ or $q$ tends to infinity (excluding the special case $i=0$ with $G = {\rm L}_2(q)$ and $q$ odd) and this is how we establish the asymptotic statements in parts (ii) and (iii) of Theorem \ref{t:class_main}.

In order to derive suitable upper bounds $\rho_i \leqs f_i(n,q)$, we need to study the conjugacy classes of prime order elements in $G$. With this aim in mind, fix a prime $p \in \pi'(G)$ and let $x \in G$ be an element of order $p$. We may view $x$ as an element of ${\rm PGL}(V)$, where we recall that $V$ is the natural module for $G$. In particular, we can write $x = \hat{x}Z$, where $\hat{x} \in {\rm GL}(V)$ and $Z$ is the centre of ${\rm GL}(V)$, and it will be useful to define the positive integer
\begin{equation}\label{e:nu}
\nu(x) = \min\{\dim [\bar{V}, \lambda \hat{x}] \,:\, \l \in K^{\times} \},
\end{equation}
where $K$ is the algebraic closure of $\mathbb{F}_q$ and $\bar{V} = V \otimes K$. Notice that $\nu(x)$ coincides with the codimension of the largest eigenspace of $\hat{x}$ on $\bar{V}$, which explains why it is sometimes referred to as the \emph{support} of $x$, in analogy with the more familiar definition of support for permutations. Bounds on $|x^G|$ in terms of $\nu(x)$ are given in \cite[Section 3]{FPR2}. In addition, we refer the reader to \cite[Chapter 3]{BGiu} for detailed information on the conjugacy classes of semisimple elements in $G$ of prime order.

We now consider each family of classical groups in turn, beginning with the linear groups. Throughout this section, we will repeatedly refer to the notation defined in \eqref{e:defs}, \eqref{e:pm}, \eqref{e:alpha} and \eqref{e:nu}.

\subsection{Linear groups}\label{ss:linear}

In this section we assume $G = {\rm L}_n(q)$ is a linear group. Note that the low-dimensional groups with $n \in \{2,3\}$ require special attention and they will 
be handled separately in Lemmas \ref{l:l3} and \ref{l:l2}. 

\begin{lem}\label{l:ln}
The conclusion to Theorem \ref{t:class_main} holds when $G = {\rm L}_n(q)$ and  $n \geqs 4$.
\end{lem}

\begin{proof} 
Set $d = (n,q-1)$ and observe that 
\[
|G| = \frac{1}{d}q^{\frac{1}{2}n(n-1)}\prod_{i=2}^n(q^i-1), \;\;  |N_G(H_r)| = \frac{1}{d}q^{\frac{1}{2}n(n-1)}(q-1)^{n-1}.
\]
Therefore, the desired bound $\Sigma(G)<1$ holds if 
\begin{equation}\label{e:ln_prob}
\what{\b}(G) \leqs \left(1-q^{-1}\right)^{n-1}.
\end{equation}

Fix a prime $p \in \pi'(G)$ and let $x \in G$ be an element of order $p$. Define the integer $\nu(x)$ as in \eqref{e:nu}. We begin by handling the groups with $n \geqs 5$.

\vs

\noindent \emph{Case 1. $n \geqs 5$}

\vs

First assume $p = 2$, so $q$ is odd. By inspecting \cite[Table B.1]{BGiu} we observe that
\[
|x^G| \geqs \frac{|{\rm GL}_n(q)|}{|{\rm GL}_{n-2}(q)||{\rm GL}_{2}(q)|} > \frac{1}{2}q^{4n-8} = b_2 
\]
if $\nu(x) \geqs 2$, otherwise
\[
|x^G| = \frac{|{\rm GL}_n(q)|}{|{\rm GL}_{n-1}(q)||{\rm GL}_{1}(q)|} = \frac{q^{n-1}(q^n-1)}{q-1} = b_1.
\]

Suppose $q \equiv 1 \imod{4}$. By applying Lemmas \ref{l:nt}(i) and \ref{l:factorial} we compute 
\[
|H_2| = \frac{1}{(d)_2}((q-1)_2)^{n-1}(n!)_2 \leqs 2^{n-1}(q-1)^{n-1} = a_2
\]
and we deduce that $H_2 < L < G$, where $L$ is a subgroup of type ${\rm GL}_1(q) \wr S_n$. That is to say, $L$ is the stabiliser in $G$ of a direct sum decomposition $V = V_1 \oplus \cdots \oplus V_n$ of the natural module for $G$, where each $V_i$ is a $1$-dimensional subspace. If $\nu(x) = 1$ then replacing $x$ by a suitable conjugate if necessary, we may assume $x$ lifts to a diagonal matrix $\hat{x} = {\rm diag}(-I_{n-1},I_1)$ in ${\rm GL}_n(q)$. And then by counting the number of such involutions in the subgroup ${\rm GL}_1(q) \wr S_n < {\rm GL}_n(q)$, we deduce that there are no more than
\[
a_1 = \binom{n}{1}+\binom{n}{2}|{\rm GL}_1(q)| = n+\frac{1}{2}n(n-1)(q-1)
\]
involutions $x \in H_2$ with $\nu(x) =1$. So Lemma \ref{l:est} implies that 
$\rho_0 < a_1^2b_1^{-1} + a_2^2b_2^{-1}$, where the $a_i$ and $b_i$ terms are defined by the equations they appear in above.

Now assume $q \equiv 3 \imod{4}$. If $n = 2\ell$ is even then 
\[
|H_2| = \frac{1}{4}((q-1)_2(q^2-1)_2)^{\ell}(\ell!)_2 \leqs 2^{3\ell-2}(q+1)^\ell = a_2
\]
and we note that $(q-1)_2 \leqs (n)_2$, which implies that $\nu(x) \geqs 2$ (see \cite[Table B.1]{BGiu}). It follows that $\rho_0 < a_2^2b_2^{-1}$. On the other hand, if $n = 2\ell+1$ is odd, then 
\[
|H_2| = ((q-1)_2(q^2-1)_2)^{\ell}(\ell!)_2 \leqs 2^{3\ell}(q+1)^\ell = a_2
\]
and $H_2$ is contained in a subgroup $L<G$ of type $({\rm GL}_2(q) \wr S_\ell) \oplus {\rm GL}_1(q)$. By considering the number of involutions $x \in L$ with $\nu(x) = 1$, we deduce that $\rho_0 < a_1^2b_1^{-1} + a_2^2b_2^{-1}$, where
\[
a_1 = 1 + \binom{\ell}{1}\frac{|{\rm GL}_2(q)|}{|{\rm GL}_1(q)|^2} = 1+\ell q(q+1).
\]

To complete the argument for $n \geqs 5$ we may assume $p$ is odd, in which case $p \in \mathcal{P}_m$ for some positive integer $m \leqs n$. Set $k = \lfloor n/m \rfloor$ and write $\gamma = (4,q-1)$ and $\delta = (4,q+1)$.

\renewcommand{\arraystretch}{1.4}

First assume $m \geqs 3$, in which case $p \geqs 5$. By inspecting \cite[Section 3.2.1]{BGiu}, we see that 
\[
|x^G| \geqs \left\{ \begin{array}{ll}
\frac{|{\rm GL}_n(q)|}{|{\rm GL}_{n-3}(q)||{\rm GL}_1(q^3)|} & \mbox{if $n \ne 6$} \\
\frac{|{\rm GL}_6(q)|}{|{\rm GL}_{2}(q^3)|} & \mbox{if $n=6$}
\end{array}\right.
\]
and thus $|x^G| > \frac{1}{2}q^{6n-12} = b_3$ for all $n \geqs 5$. In addition, by applying Lemmas \ref{l:nt}(ii) and \ref{l:factorial}, we observe that
\[
|H_p| = \prod_{i=1}^{k}(q^{mi}-1)_p = ((q^m-1)_p)^{k}(k!)_p < q^{5n/4} = a_3
\]
since $(k!)_p < p^{k/(p-1)} < q^{m(n/4m)} = q^{n/4}$. This yields $\what{Q}_p(G) < a_3^2b_3^{-1}$. Since there are $n-2$ possibilities for $m$ in the range $3 \leqs m \leqs n$, and since each $q^m-1$ has at most $n\log q$ distinct prime divisors, we deduce that  
\[
\rho_3 < n(n-2)\log q \cdot a_3^2b_3^{-1}.
\]

\renewcommand{\arraystretch}{1}

Next assume $p \in \mathcal{P}_2$, so $p$ divides $(q+1)/\delta$. Here 
\[
|x^G| \geqs \frac{|{\rm GL}_n(q)|}{|{\rm GL}_{n-2}(q)||{\rm GL}_1(q^2)|} >\frac{1}{2}q^{4n-6} = b_4
\]
and 
\[
|H_p| = ((q+1)_p)^{k}(k!)_p \leqs \left(\frac{1}{\delta}(q+1)\right)^{3n/4} = a_4,
\]
which means that $\what{Q}_p(G) < a_4^2b_4^{-1}$ and $\rho_2 < \log((q+1)/\delta) \cdot a_4^2b_4^{-1}$.

Finally, suppose $p \in \mathcal{P}_1$. Here $p$ divides $(q-1)/\gamma$ and 
\[
|H_p| = \frac{1}{(d)_p}((q-1)_p)^{n-1}(n!)_p \leqs \left(\frac{1}{\gamma}(q-1)\right)^{3n/2-1} = a_6.
\]
In particular, we notice that $H_p$ is contained in a subgroup $L<G$ of type ${\rm GL}_1(q) \wr S_n$. If $\nu(x) \geqs 2$, then $|x^G| > \frac{1}{2}q^{4n-8} = b_6$. Now assume $\nu(x) = 1$, which means that some $G$-conjugate of $x$ lifts to a diagonal matrix of the form $\hat{x} = {\rm diag}(\l, I_{n-1}) \in {\rm GL}_n(q)$, where $\l \in \mathbb{F}_{q}^{\times}$ has order $p$. It follows that 
\[
|x^G| = \frac{|{\rm GL}_n(q)|}{|{\rm GL}_{n-1}(q)||{\rm GL}_1(q)|} = \frac{q^{n-1}(q^n-1)}{q-1} = b_5
\]
and by counting in $L$ we deduce that there are no more than 
\[
\binom{n}{1}(p-1) \leqs n((q-1)/\gamma-1) = a_5
\]
such elements in $H_p$. This yields $\rho_1 < \log((q-1)/\gamma) \cdot \left(a_5^2b_5^{-1} + a_6^2b_6^{-1}\right)$.

It is clear in every case that the given upper bound on $\rho_i$ tends to $0$ as $n$ or $q$ tends to infinity, so the asymptotic statements in parts (ii) and (iii) of Theorem \ref{t:class_main} hold when $G = {\rm L}_n(q)$ and $n \geqs 5$. In addition, if $n \geqs 7$ then it is routine to check that the given estimates imply that the inequality in \eqref{e:ln_prob} is satisfied unless $(n,q)$ is one of the following:
\begin{equation}\label{e:list1}
(7,2), \, (7,3), \, (8,2), \, (8,3), \, (9,2), \, (9,3), \, (10,2), \, (10,3), \, (12,3).
\end{equation}

Suppose $(n,q) = (7,2)$. If $p \in \mathcal{P}_m$ with $m \geqs 3$ then $|H_p| \leqs 127$ and we note that $|\pi'(G)| = 5$, $\mathcal{P}_2 = \{3\}$, $\mathcal{P}_1 = \emptyset$ and $|H_3| = 3^4$. Therefore,
\[
\what{\b}(G) < 4a_3^2b_3^{-1} + a_4^2b_4^{-1} < 2^{-6}
\]
where $a_3 = 127$, $a_4 = 3^4$ and $b_3,b_4$ are defined as above. Similarly, if $(n,q) = (7,3)$ then $|\pi'(G)| = 6$, $\mathcal{P}_1 = \mathcal{P}_2 =  \emptyset$ and $|H_2| = 2^{13}$. In addition, $|H_p| \leqs 1093$ for all odd primes $p \in \pi'(G)$ and thus  
\[
\what{\b}(G) < a_1^2b_1^{-1} + a_2^2b_2^{-1} + 5a_3^2b_3^{-1} < \left(\frac{2}{3}\right)^6,
\]
where $a_2 = 2^{13}$, $a_3 = 1093$ and $a_1,b_1,b_2,b_3$ are defined as above. All of the other cases in \eqref{e:list1} can be handled in an entirely similar fashion.

Now assume $n \in \{5,6\}$. Here we observe that $|H_p| \leqs c((q-1)/\gamma)^{n-1}$ for all $p \in \mathcal{P}_1$, where $c = 5$ if $n=5$ and $c = 9$ if $n=6$ (since $(5!)_p \leqs 5$ and $(6!)_p \leqs 3^2$ for all odd primes $p$). Then by setting $a_6 = c((q-1)/\gamma)^{n-1}$, one can check that our previous bounds are sufficient unless $(n,q)$ is one of the following:
\[
(5,2), \, (5,3), \, (5,4), \, (5,5),\, (6,2),\, (6,3). 
\]
All of these special cases can be handled using {\sc Magma}. For example, if $(n,q) = (5,2)$ then we compute 
\[
\Sigma(G) = \frac{4483687}{4999680} < 1,
\]
where $\Sigma(G)$ is defined in \eqref{e:sigma}.

\vs

\noindent \emph{Case 2. $n=4$}

\vs

Now assume $n=4$. If $p=2$ and $q \equiv 1 \imod{4}$, then 
$\rho_0 \leqs a_1^2b_1^{-1} + a_2^2b_2^{-1}$, where 
\[
a_1 = \binom{4}{1}+\binom{4}{2}(q-1), \;\; b_1 = \frac{|{\rm GL}_4(q)|}{|{\rm GL}_3(q)||{\rm GL}_1(q)|} = q^3(q+1)(q^2+1)
\]
and
\[
a_2 = 2(q-1)^3, \;\; b_2 =  \frac{|{\rm GL}_4(q)|}{2|{\rm GL}_2(q^2)|} = \frac{1}{2}q^4(q-1)(q^3-1),
\]
noting that $|H_2| \leqs a_2$. Similarly, if $p=2$ and $q \equiv 3 \imod{4}$, then $\nu(x) = 2$ (see \cite[Table B.1]{BGiu}) and Lemma \ref{l:est} implies that $\rho_0 \leqs a_2^2b_2^{-1}$, where $|H_2| \leqs 8(q+1)^2 = a_2$ and $b_2$ is defined as above.

Now suppose $p \geqs 3$ is odd and define $\gamma,\delta$ as before. Note that $p \in \mathcal{P}_m$ with $m \in \{1,2,3,4\}$. If $m \in \{3,4\}$ then $|H_p| \leqs q^2+q+1 = a_3$ and 
\[
|x^G| \geqs \frac{|{\rm GL}_4(q)|}{|{\rm GL}_1(q^4)|} = q^6(q-1)(q^2-1)(q^3-1) = b_3,
\]
which means that
\[
\rho_3 \leqs 2\log(q^2+q+1) \cdot a_3^2b_3^{-1}.
\]
Similarly, if $m=2$ then $|H_p| \leqs ((q+1)/\delta)^2 = a_4$ and 
\[
|x^G| \geqs \frac{|{\rm GL}_4(q)|}{|{\rm GL}_2(q^2)|} = q^4(q-1)(q^3-1) = b_4,
\]
which yields $\rho_2 \leqs \log((q+1)/\delta) \cdot a_4^2b_4^{-1}$. 

Finally, suppose $m=1$, so $p$ divides $(q-1)/\gamma$ and $H_p$ is contained in a subgroup of type ${\rm GL}_1(q) \wr S_4$. In particular, $|H_p| \leqs 3((q-1)/\gamma)^3 = a_6$. If $\nu(x) = 1$ then 
\[
|x^G| = \frac{|{\rm GL}_4(q)|}{|{\rm GL}_3(q)||{\rm GL}_1(q)|} = q^3(q+1)(q^2+1) = b_5
\]
and by arguing as in Case 1 we see that there are no more than $a_5 = 4((q-1)/\gamma-1)$ such elements in $H_p$. For $\nu(x) \geqs 2$ we have
\[
|x^G| \geqs \frac{|{\rm GL}_4(q)|}{|{\rm GL}_2(q)|^2} = q^4(q^2+1)(q^2+q+1) = b_6
\]
and it follows that $\rho_1 \leqs \log((q-1)/\gamma)\cdot(a_5^2b_5^{-1} + a_6^2b_6^{-1})$.

As before, it is clear that each upper bound on $\rho_i$ tends to $0$ as $q$ tends to infinity. In addition, for $q \geqs 4$ we find that our bounds are good enough to give the desired inequality in \eqref{e:ln_prob} with $n=4$. In the two remaining cases we compute $\Sigma(G) = 9799/10080$ when $q =2$, whereas $\Sigma(G) = 388757/379080$ when $q = 3$. So this gives the result for $q = 2$, but further work is needed when $q = 3$. In the latter case, we have $\pi(G) = \{2,3,5,13\}$ and for each $p \in \pi(G)$ we can use {\sc Magma} to compute
\[
Q_p(G) = 1 - \frac{s|H_p|}{|G:H_p|},
\]
where $s$ is the number of $(H_p,H_p)$ double cosets in $G$ of size $|H_p|^2$. This yields
\[
Q_2(G) = \frac{829}{3645},\; Q_3(G) = \frac{1351}{2080}, \; Q_5(G) = \frac{1}{75816},\; Q_{13}(G) =  \frac{1}{155520}
\]
and thus $\a(G) < 1$ as required.
\end{proof}

To complete the proof of Theorem \ref{t:class_main} for linear groups, it just remains for us to handle ${\rm L}_3(q)$ and ${\rm L}_2(q)$.  

\begin{lem}\label{l:l3}
The conclusion to Theorem \ref{t:class_main} holds when $G = {\rm L}_3(q)$.
\end{lem}

\begin{proof}
First observe that $\Sigma(G)<1$ if  
\begin{equation}\label{e:l3_prob}
\what{\b}(G) < \frac{q^3}{(q+1)(q^2+q+1)}.
\end{equation}
Fix a prime $p \in \pi'(G)$ and set $\gamma = (4,q-1)$ and $\delta = (4,q+1)$ as before. Let $x \in G$ be an element of order $p$.

First assume $p=2$, so $q$ is odd and $|x^G| = q^2(q^2+q+1) = b_1$ for every involution $x \in G$. If $q \equiv 1 \imod{4}$ then $|H_2| = 2((q-1)_2)^2$ and $H_2$ is contained in a subgroup $L<G$ of type ${\rm GL}_1(q) \wr S_3$. By counting in $L$, we see that $i_2(H_2) \leqs 3+3(q-1) = 3q= a_1$. And for 
$q \equiv 3 \imod{4}$ we have $|H_2| \leqs 4(q+1) = a_1$. It follows that $\rho_0 \leqs a_1^2b_1^{-1}$ for all odd $q$.

Now assume $p \in \mathcal{P}_m$ with $m \in \{1,2,3\}$. If $m=3$ then 
\[
|H_p| = (q^2+q+1)_p \leqs q^2+q+1 = a_2,\;\; |x^G| = q^3(q-1)(q^2-1) = b_2
\]
and thus $\rho_3 \leqs a_2^2b_2^{-1}\log a_2$. Similarly, we have $\rho_2 \leqs a_3^2b_3^{-1}\log a_3$, 
where $a_3 = (q+1)/\delta$ and $b_3 = q^3(q^3-1)$. Finally, suppose $m=1$. Here $|H_p| \leqs 3((q-1)/\gamma)^2 = a_5$ and we note that $H_p$ is contained in a subgroup of type ${\rm GL}_1(q) \wr S_3$. If $\nu(x) = 1$ then $|x^G| = q^2(q^2+q+1) = b_4$ and there are at most $3((q-1)/\gamma-1) = a_4$ of these elements in $H_p$. Otherwise, if $\nu(x) = 2$ then $|x^G| \geqs \frac{1}{3}q^3(q+1)(q^2+q+1) = b_5$ and we deduce that 
\[
\rho_1 \leqs \log((q-1)/\gamma) \cdot (a_4^2b_4^{-1} + a_5^2b_5^{-1}).
\]

The above estimates show that $\b(G) \to 0$ as $q \to \infty$, and they also imply that the inequality in \eqref{e:l3_prob} holds if $q \geqs 9$. For $3 \leqs q \leqs 8$, we can use {\sc Magma} to check that $\Sigma(G) < 1$. Finally, for $q = 2$ we compute $\Sigma(G) = 121/84$, so this case requires further attention. Here $\pi(G) =\{2,3,7\}$ and using {\sc Magma} we get 
\[
Q_2(G) = \frac{13}{21},\; Q_3(G) = \frac{1}{28}, \; Q_7(G) = \frac{1}{8},
\]
which implies that $\a(G)<1$.
\end{proof}

\begin{lem}\label{l:l2}
The conclusion to Theorem \ref{t:class_main} holds when $G = {\rm L}_2(q)$.
\end{lem}

\begin{proof}
First observe that $|G| = \frac{1}{d}q(q^2-1)$ and $|N_G(H_r)| = \frac{1}{d}q(q-1)$, where $d = (2,q-1)$, so we have $\Sigma(G)<1$ if
\begin{equation}\label{e:l2_prob}
\what{\b}(G)  < \frac{q}{q+1}.
\end{equation}

Fix a prime $p \in \pi'(G)$. For $p=2$ we have $H_2 = D_{(q-\e)_2}$, where $q \equiv \e \imod{4}$ and $\e = \pm 1$, so $i_2(H_2) \leqs \frac{1}{2}(q+3) = a_1$ and we note that $|x^G| \geqs \frac{1}{2}q(q-1) = b_1$ for all involutions $x \in G$. It follows that
\[
\rho_0 \leqs a_1^2b_1^{-1}  = \frac{(q+3)^2}{2q(q-1)}.
\]

Notice that this estimate does not imply that $Q_2(G) \to 0$ as $q$ tends to infinity, which explains why $G = {\rm L}_2(q)$ is excluded in part (iii) of Theorem \ref{t:class_main}. For example, if $q = 2^k-1$ is a Mersenne prime, then $H_2 = D_{q+1}$ and by appealing to \cite[Lemma 7.9]{BH20} we compute
\[
Q_2(G) = 1 - \frac{s|H_2|}{|G:H_2|} = \frac{1}{2} + \frac{q+3}{2q(q-1)},
\]
where $s = (q-3)/4$ is the number of regular orbits of $H_2$ on $G/H_2$. In particular, we observe that $Q_2(G) > \frac{1}{2}$ for every Mersenne prime $q$ (and of course, we expect that there are infinitely many such primes).

Now assume $p \geqs 3$ and let $x \in G$ be an element of order $p$. Set $\gamma = (4,q-1)$ and $\delta = (4,q+1)$. First suppose $p$ divides $q-1$. Now $G$ has $(p-1)/2$ distinct conjugacy classes of such elements, with $|x^G| = q(q+1)$ and $|x^G \cap H_p| = 2$ for all $x \in G$ of order $p$, whence 
\[
\what{Q}_p(G) = \frac{1}{2}(p-1) \cdot \frac{4}{q(q+1)} = \frac{2(p-1)}{q(q+1)} \leqs \frac{2(q-2)}{q(q+1)}.
\]
Similarly, if $p$ divides $q+1$, then $|x^G| = q(q-1)$, $|x^G \cap H_p| = 2$ and we get
\[
\what{Q}_p(G) = \frac{1}{2}(p-1) \cdot \frac{4}{q(q-1)} \leqs \frac{2}{q-1}.
\]
It follows that
\[
\gamma(G) \leqs \log((q-1)/\gamma) \cdot \frac{2(q-2)}{q(q+1)} + \log((q+1)/\delta) \cdot \frac{2}{q-1}
\]
and thus $\gamma(G) \to 0$ as $q \to \infty$. In addition, we get
\[
\what{\b}(G) \leqs (1-\delta_{2,r}) \frac{(q+3)^2}{2q(q-1)} + \log((q-1)/\gamma) \cdot \frac{2(q-2)}{q(q+1)} + \log((q+1)/\delta) \cdot \frac{2}{q-1},
\]
which implies that the inequality in \eqref{e:l2_prob} holds for all $q >37$. The remaining groups with $5 \leqs q \leqs 37$ can be checked using {\sc Magma}. Indeed, for $q \ne 7$ we compute $\Sigma(G)<1$, while for $q = 7$ we have $G \cong {\rm L}_3(2)$ and so this case has already been handled in Lemma \ref{l:l3}.
\end{proof}

\subsection{Unitary groups}\label{ss:unit}

We now turn to the proof of Theorem \ref{t:class_main} for the unitary groups. So throughout this section we assume $G = {\rm U}_n(q)$ with $n \geqs 3$ and $(n,q) \ne (3,2)$. Our main result is the following (note that the $3$-dimensional groups are handled separately in Lemma \ref{l:u3}).

\begin{lem}\label{l:un}
The conclusion to Theorem \ref{t:class_main} holds when $G = {\rm U}_n(q)$ and  $n \geqs 4$.
\end{lem}

\begin{proof}
Set $d = (n,q+1)$ and first observe that 
\[
|N_G(H_r):H_r| = \frac{1}{d}(q^2-1)^{\lfloor (n-1)/2 \rfloor}(q-1)^e,
\]
where $e = 1$ if $n$ is even, otherwise $e = 0$. It follows that $\Sigma(G)<1$ if the bound in \eqref{e:ln_prob} is satisfied.

\vs

\noindent \emph{Case 1. $n \geqs 5$}

\vs

Fix a prime $p \in \pi'(G)$ and let $x \in G$ be an element of order $p$. 
First assume $p=2$, so $q$ is odd and we have 
\[
|x^G| = \frac{|{\rm GU}_n(q)|}{|{\rm GU}_{n-1}(q)||{\rm GU}_1(q)|} = \frac{q^{n-1}(q^n-(-1)^n)}{q+1} = b_1
\]
if $\nu(x) = 1$ (see \cite[Table B.1]{BGiu}). Similarly, $|x^G|>\frac{1}{2}(q+1)^{-1}q^{4n-7} = b_2$ if $\nu(x) = 2$ and we have $|x^G| > \frac{1}{4}(q+1)^{-1}q^{6n-17} = b_3$ if $\nu(x) \geqs 3$. There are now four cases to consider, according to the parity of $n$ and the congruence class of $q$ modulo $4$.

Suppose $n=2\ell+1$ is odd and $q \equiv 1 \imod{4}$. By appealing to Lemma \ref{l:nt}(i) we compute
\[
|H_2| = ((q+1)_2(q^2-1)_2)^{\ell}(\ell!)_2 \leqs 2^{3\ell}(q-1)^{\ell} = a_2
\]
and it follows that $H_2 <L<G$, where $L$ is of type $({\rm GU}_2(q) \wr S_{\ell}) \perp {\rm GU}_1(q)$. Here the notation indicates that $L$ is the stabiliser in $G$ of an orthogonal decomposition
\[
V = U \perp W = U_1 \perp \cdots \perp U_{\ell} \perp W
\]
of the natural module for $V$, where each $U_i$ is a non-degenerate $2$-space and $W$ is a non-degenerate $1$-space. By counting in $L$, we deduce that there are at most
\[
a_1 = 1+\binom{\ell}{1}\frac{|{\rm GU}_2(q)|}{|{\rm GU}_1(q)|^2} = 1+\ell q(q-1)
\]
involutions $x \in H_2$ with $\nu(x) = 1$, so Lemma \ref{l:est} yields $\rho_0 < a_1^2b_1^{-1} + a_2^2b_2^{-1}$.

Next assume $n=2\ell+1$ and $q \equiv 3 \imod{4}$. Here  
\[
|H_2| = ((q+1)_2)^{n-1}(n!)_2 \leqs 2^{n-1}(q+1)^{n-1} = a_3
\]
and we see that $H_2 < L<G$ with $L$ of type ${\rm GU}_1(q) \wr S_n$. Then by working in $L$, we calculate that the number of involutions $x \in H_2$ with $\nu(x) = 1$ is at most 
\begin{equation}\label{e:psua1}
a_1 = \binom{n}{1}+\binom{n}{2}|{\rm GU}_1(q)| = n+\frac{1}{2}n(n-1)(q+1).
\end{equation}
Similarly, there are at most
\begin{equation}\label{e:psua2}
\begin{aligned}
a_2 & = \binom{n}{2} + \binom{n}{2}\binom{n-2}{1}|{\rm GU}_1(q)| + \frac{n!}{(n-4)!2!2^2}|{\rm GU}_1(q)|^2 \\
& = \frac{1}{2}n(n-1)\left(1+(n-2)(q+1)+\frac{1}{4}(n-2)(n-3)(q+1)^2\right)
\end{aligned}
\end{equation}
involutions in $H_2$ with $\nu(x) = 2$. This implies that 
\begin{equation}\label{e:rho0}
\rho_0 < a_1^2b_1^{-1} + a_2^2b_2^{-1} + a_3^2b_3^{-1}.
\end{equation}

Now assume $n = 2\ell$ is even. If $q \equiv 1 \imod{4}$ then $\nu(x) \geqs 2$ (see \cite[Table B.1]{BGiu}) and 
\[
|H_2| = \frac{1}{4}((q+1)_2(q^2-1)_2)^{\ell}(\ell!)_2 \leqs 2^{3\ell-2}(q-1)^{\ell} = a_2,
\]
whence Lemma \ref{l:est} implies that $\rho_0 < a_2^2b_2^{-1}$. Finally, suppose $q \equiv 3 \imod{4}$. Here
\[
|H_2| = \frac{1}{(d)_2}((q+1)_2)^{n-1}(n!)_2 \leqs 2^{n-1}(q+1)^{n-1} = a_3
\]
and thus $H_2 < L<G$ with $L$ of type ${\rm GU}_1(q) \wr S_n$. Then as above, we calculate that $H_2$ contains at most $a_1$ (respectively, $a_2$) involutions with $\nu(x)=1$ (respectively, $\nu(x) = 2$), where $a_1$ and $a_2$ are defined in \eqref{e:psua1} and \eqref{e:psua2}, and we conclude that \eqref{e:rho0} holds.

\renewcommand{\arraystretch}{1.4}

Now assume $p \geqs 3$ is odd, so $p \in \mathcal{P}_{m}$ with $m \leqs 2n$ . Note that either $m \leqs \lfloor n/2 \rfloor$ is odd, or $m \leqs n$ is even, or $m = 2i$ with $n/2< i \leqs n$ odd. In particular, there are precisely $n$ possibilities for $m$. 

To begin with, let us assume $m \geqs 3$ and note that $p \geqs 5$. By considering the possibilities for $|C_G(x)|$ (see \cite[Section 3.3.1]{BGiu}), we first observe that  
\[
|x^G| \geqs \left\{ \begin{array}{ll}
\frac{|{\rm GU}_n(q)|}{|{\rm GU}_{n-3}(q)||{\rm GU}_1(q^3)|} & \mbox{if $n \ne 6$} \\
\frac{|{\rm GU}_6(q)|}{|{\rm GU}_{2}(q^3)|} & \mbox{if $n=6$}
\end{array}\right.
\]
and thus $|x^G| > \frac{1}{2}(q+1)^{-1}q^{6n-11} = b_4$.

\renewcommand{\arraystretch}{1}

We claim that 
\[
|H_p| \leqs (q^3+1)^{\frac{11}{24}n+\frac{1}{3}} = a_4.
\]
If $m \equiv 0 \imod{4}$, then by applying Lemmas \ref{l:nt}(ii) and \ref{l:factorial}, we compute
\[
|H_p| = \prod_{i=1}^{k}(q^{mi}-1)_p = ((q^{m/2}+1)_p)^{k}(k!)_p < (q^{m/2}+1)^{\frac{5n}{4m}} \leqs (q^2+1)^{\frac{5}{16}n} < a_4,
\]
where $k = \lfloor n/m\rfloor$. Similarly, if $m$ is odd and $k = \lfloor n/2m \rfloor$ then
\[
|H_p| = \prod_{i=1}^{k}(q^{2mi}-1)_p = ((q^{m}-1)_p)^{k}(k!)_p < q^{\frac{5}{8}n} < a_4.
\]
Finally, suppose $m \equiv 2 \imod{4}$ and set $k = \lfloor (n-1+e)/m \rfloor$ and $\ell = \lfloor (n-e)/m+1/2 \rfloor$, where $e = 1$ if $n$ is even, otherwise $e = 0$. Then
\[
|H_p| = \prod_{i=1}^{k}(q^{mi}-1)_p \cdot \prod_{i=1}^{\ell}(q^{m(2i-1)}-1)_p = ((q^{m/2}+1)_p)^{k+\ell}(k!)_p\prod_{i=1}^{\ell}(2i-1)_p
\]
and we deduce that
\[
|H_p| \leqs (q^{m/2}+1)^{\frac{5}{4}k+\ell+\frac{1}{4}(2\ell-1)} \leqs (q^3+1)^{\frac{11}{24}n+\frac{1}{3}} = a_4.
\]
This justifies the claim. In addition, since there are $n-2$ possibilities for $m$ with $m \geqs 3$ and since $|\mathcal{P}_m| < \frac{1}{2}n\log q$ for all $m$ (see Lemma \ref{l:primes}), it follows that
\[
\rho_3 < \frac{1}{2}n(n-2)\log q \cdot a_4^2b_4^{-1}.
\]

Next suppose $m=2$, so $p$ divides $(q+1)/\delta$ and  
\[
|H_p| = \frac{1}{(d)_p}((q+1)_p)^{n-1}(n!)_p \leqs ((q+1)/\delta)^{\frac{3}{2}n-1} = a_7,
\]
where $\delta = (4,q+1)$. In particular, we have $H_p < L < G$ with $L$ of type ${\rm GU}_1(q) \wr S_n$. If $\nu(x) = 1$ then 
\[
|x^G| = \frac{|{\rm GU}_n(q)|}{|{\rm GU}_{n-1}(q)||{\rm GU}_1(q)|} = \frac{q^{n-1}(q^n-(-1)^n)}{q+1} = b_5
\] 
and by counting in $L$ we deduce that there are at most 
\[
\binom{n}{1}(p-1) \leqs n((q+1)/\delta-1) = a_5
\]
of these elements in $H_p$. Similarly, if $\nu(x) = 2$ then $|x^G|>\frac{1}{2}(q+1)^{-1}q^{4n-7} = b_6$ and $H_p$ contains no more than
\[
\binom{n}{2}(p-1)^2 + \delta_{3,p}\frac{n!}{(n-3)!3} |{\rm GU}_1(q)|^2
\]
such elements, which in turn is at most
\[
a_6 = \frac{1}{2}n(n-1)\left( ((q+1)/\delta-1)^2 + \frac{2}{3}(n-2)(q+1)^2\right).
\]
Since $|x^G|>\frac{1}{2}(q+1)^{-1}q^{6n-17} = b_7$ if $\nu(x) \geqs 3$, it follows that 
\[
\rho_2 < \log((q+1)/\delta) \cdot (a_5^2b_5^{-1} + a_6^2b_6^{-1} + a_7^2b_7^{-1}).
\]

Finally, suppose $m=1$ and set $\gamma = (4,q-1)$ and $k = \lfloor n/2 \rfloor$. Then $p$ divides $(q-1)/\gamma$ and  
\[
|H_p| = ((q-1)_p)^{k}(k!)_p \leqs ((q-1)/\gamma)^{\frac{3}{4}n} = a_8.
\]
In addition, we have 
\[
|x^G| \geqs \frac{|{\rm GU}_n(q)|}{|{\rm GU}_{n-2}(q)||{\rm GL}_1(q^2)|} > \frac{1}{2}\left(\frac{q}{q+1}\right)q^{4n-6} = b_8
\]
for all $x \in G$ of order $p$, which implies that 
\[
\rho_1 < \log((q-1)/\gamma) \cdot a_8^2b_8^{-1}.
\]

For each $i$, it is clear that the above estimates $\rho_i < f_i(n,q)$ have the property that $f_i(n,q) \to 0$ as either $n$ or $q$ tends to infinity. In particular, $\b(G) \to 0$ as $|G| \to \infty$. 

So to complete the proof of the lemma for $n \geqs 5$, it just remains to show that $\a(G)<1$. To do this, it will be convenient to partition our analysis of the above bounds according to the parity of $q$.

First assume $q$ is odd. Here it is straightforward to check that the bound
\[
\what{\b}(G) < \sum_{i=0}^3 f_i(n,q)
\]
shows that the inequality in \eqref{e:ln_prob} is satisfied unless $(n,q)$ is one of the following:
\[
(5,3), \; (5,5),\; (5,7), \; (6,3), \; (7,3), \; (8,3).
\]
For $(n,q) = (5,3)$ we can use {\sc Magma} to check that $\Sigma(G)<1$. And in  each of the remaining cases, we can adjust the above estimates to give the desired result. For example, if $(n,q) = (5,5)$ then we compute $i_2(H_2) = 51$ and $\pi'(G) = \{2,3,7,13,521\}$, which means that we can set $a_2 = 51$ in our previous upper bound on $\rho_0$ and we can replace the factor $\frac{1}{2}n(n-2)\log q$ in the expression for $a_4$ by $3$. It is then straightforward to  check that the modified bound is sufficient.

Now suppose $q$ is even. The case $q=2$ will require special attention, so for now let us assume $q \geqs 4$. If $n \geqs 7$ then the bound
\[
\what{\b}(G) < \sum_{i=1}^3 f_i(n,q)
\]
is sufficient unless $(n,q) = (7,4)$. In the latter case, we have $\mathcal{P}_2 = \{5\}$ and $|H_5| = 5^7$, so we can set $a_7 = 5^7$ in the expression for $f_2(n,q)$ and then check that the corresponding bound on $\what{\b}(G)$ is sufficient.

Next assume $n \in \{5,6\}$ and $q \geqs 4$. If $n = 6$ then $m \in \{1,2,3,4,6,10\}$ and for $m \geqs 3$ we observe that $|H_p| \leqs q^4-q^3+q^2-q+1 = a_4$ and $p$ divides
\[
(q^2+1)(q^2+q+1)(q^2-q+1)(q^4-q^3+q^2-q+1) < q^{10},
\]
which implies that $\rho_3 < 10\log q \cdot a_4^2b_4^{-1}$, where $b_4$ is defined as above. Similarly, if $m=2$ then $|H_p| \leqs 5(q+1)^5$, so we can set $a_7 = 5(q+1)^5$ in the definition of $f_2(n,q)$. With these modifications, it is easy to check that the corresponding upper bound on $\what{\b}(G)$ is sufficient for all $q \geqs 4$.

The case $n=5$ is very similar. Here $\rho_3 < 8\log q \cdot a_4^2b_4^{-1}$, where $a_4 = q^4-q^3+q^2-q+1$ and 
$b_4$ is defined as above. In addition, we can set $a_7 = c(q+1)^4$ in the definition of $f_2(n,q)$, where $c=3$ if $q \equiv -1 \imod{3}$, otherwise $c=1$. The reader can then check that the improved upper bound on $\what{\b}(G)$ is sufficient for all $q \geqs 8$. Finally, if $(n,q) = (5,4)$ then $\pi'(G) = \{3,5,13,17,41\}$ and thus
\[
\what{\b}(G) < 3a_4^2b_4^{-1} + \sum_{i=5}^8a_i^{2}b_i^{-1} < \left(\frac{3}{4}\right)^4
\]
as required.

Finally, let us assume $q = 2$. If $n \in \{5,6,7,8\}$ then we can use {\sc Magma} to check that $\Sigma(G)<1$. For example, for $n=5$ we get
\[
\Sigma(G) = \frac{6566723}{6842880}.
\]
So for the remainder, we may assume $n \geqs 9$. To resolve this case, we need to improve the upper bound $\rho_2 \leqs f_2(n,2)$ presented above, noting that $\mathcal{P}_2 = \{3\}$ and $H_3$ is contained in a subgroup $L$ of type ${\rm GU}_1(2) \wr S_n$. 

With this goal in mind, let $x \in G$ be an element of order $3$ and observe that 
\[
|x^G| \geqs \frac{|{\rm GU}_n(2)|}{|{\rm GU}_{n-4}(2)||{\rm GU}_4(2)|}>\frac{1}{2}\left(\frac{q}{q+1}\right)q^{8n-32} = b_9
\]
if $\nu(x) \geqs 4$. Now, if $\nu(x) = 3$ then $|x^G|>\frac{1}{2}(q+1)^{-1}q^{6n-17} = b_7$ and by counting in $L$ we deduce that there are no more than 
\[
a_7  = 2^3\binom{n}{3} + \frac{n!}{(n-3)!3}|{\rm GU}_1(2)|^2 \cdot 2\binom{n-3}{1}  = \frac{2}{3}n(n-1)(n-2)(9n-25) 
\]
such elements in $H_3$. Since $|H_3| < 3^{3n/2-1} = a_9$, we deduce that
\[
\rho_2 < a_5^2b_5^{-1}+a_6^2b_6^{-1}+a_7^2b_7^{-1}+a_9^2b_9^{-1},
\]
where the terms $a_5,b_5,a_6$ and $b_6$ are defined as in the main argument above. By combining this revised estimate for $\rho_2$ with our previous bound on $\rho_3$, the reader can check that $\what{\b}(G)<1$ for all $n \geqs 14$. 

Finally, suppose $9 \leqs n \leqs 13$. Recall that
\[
\frac{|H_2||N_G(H_2)|}{|G|} = \frac{2^{n(n-1)/2}3^{\lfloor (n-1)/2\rfloor}}{\prod_{i=2}^n(2^i-(-1)^i)}=:f(n)
\]
and it suffices to show that $\what{\b}(G) < f(n)$. One checks that our previous upper bound on $\what{\b}(G)$ is sufficient if $n \in \{11,12,13\}$, so we may assume $n =9$ or $10$. Here we can improve the upper bound on $\rho_2$ by setting 
\[
a_9 = i_3(H_3) = \left\{ \begin{array}{ll}
104246  & \mbox{if $n=10$} \\
28430 & \mbox{if $n=9$,}
\end{array}\right.
\]
where we have computed the exact values using {\sc Magma}. One can now check that the revised estimate for $\what{\b}(G)$ is good enough.

\vs

\noindent \emph{Case 2. $n=4$}

\vs

To complete the proof of the lemma we may assume $n=4$. In view of part (i) of Theorem \ref{t:class_main}, we may assume $q \geqs 3$. And with the aid of {\sc Magma} we check that $\Sigma(G) < 1$ when $q \in \{3,4,5\}$, so we can assume $q \geqs 7$. Note that it suffices to show that
\begin{equation}\label{e:u4_prob}
\what{\b}(G) < \frac{|H_r||N_G(H_r)|}{|G|} = \frac{q^6}{(q^3+1)(q^2+1)(q+1)}.
\end{equation}
Fix a prime $p \in \pi'(G)$ and let $x \in G$ be an element of order $p$.

First assume $p = 2$. If $q \equiv 1 \imod{4}$ then $|H_2| \leqs 8(q-1)^2 = a_2$ and $\nu(x) = 2$, so 
\[
|x^G| \geqs \frac{|{\rm GU}_4(q)|}{2|{\rm GU}_2(q)|^2} = \frac{1}{2}q^4(q^2-q+1)(q^2+1) = b_2
\]
and $\rho_0 \leqs a_2^2b_2^{-1}$. Now assume $q \equiv 3 \imod{4}$. Here $|H_2| = 2((q+1)_2)^3 \leqs 2(q+1)^3 = a_2$ and $H_2$ is contained in a subgroup $L$ of type ${\rm GU}_1(q) \wr S_4$. If $\nu(x) = 1$ then $|x^G| = q^3(q-1)(q^2+1) = b_1$ and by working in $L$ we calculate that there are at most $a_1 = 6q+10$ such elements in $H_2$, so it follows that $\rho_0 \leqs a_1^2b_1^{-1} + a_2^2b_2^{-1}$.

Now assume $p$ is odd, so $p \in \mathcal{P}_m$ with $m \in \{1,2,4,6\}$. If $m \in \{4,6\}$ then $|H_p| \leqs q^2+1 = a_3$ and
\[
|x^G| \geqs \frac{|{\rm GU}_4(q)|}{|{\rm GU}_1(q)||{\rm GU}_1(q^3)|} = q^4(q^2-1)(q^4-1) = b_3,
\]
which means that $\rho_3 \leqs 2\log(q^2+1) \cdot a_3^2b_3^{-1}$. 

Now suppose $m=2$, so $p$ divides $(q+1)/\delta$ with $\delta = (4,q+1)$. Here
\[
|H_p| = ((q+1)_p)^3(4!)_p \leqs 3((q+1)/\delta)^3 = a_5
\]
and we note that $H_p<L<G$, where $L$ is a subgroup of type ${\rm GU}_1(q) \wr S_4$. If $\nu(x) \geqs 2$ then 
\[
|x^G| \geqs \frac{|{\rm GU}_4(q)|}{|{\rm GU}_2(q)|^2} = q^4(q^2-q+1)(q^2+1) = b_5,
\]
otherwise $|x^G| = q^3(q-1)(q^2+1) = b_4$ and by counting in $L$ we deduce that $H_p$ contains at most
$4(p-1) \leqs 4((q+1)/\delta -1) = a_4$
such elements. This implies that 
\[
\rho_2 \leqs \log((q+1)/\delta) \cdot (a_4^2b_4^{-1} + a_5^2b_5^{-1}).
\]
Finally, if $m=1$ then $|H_p| \leqs ((q-1)/\gamma)^2 = a_6$ and 
\[
|x^G| \geqs \frac{|{\rm GU}_4(q)|}{|{\rm GL}_2(q^2)|} = q^4(q+1)(q^3+1) = b_6,
\]
where $\gamma = (4,q-1)$, and we deduce that $\rho_1 \leqs \log((q-1)/\gamma) \cdot a_6^2b_6^{-1}$.

Bringing all of these estimates together, it is straightforward to check that the inequality in \eqref{e:u4_prob} holds for all $q \geqs 5$ with $q \ne 8$. And for $q=8$ we have $\mathcal{P}_2 = \{3\}$, so we can replace the factor $\log((q+1)/\delta)$ in the upper bound on $\rho_2$ by $1$, and then it is easy to check that the corresponding upper bound on $\what{\b}(G)$ is sufficient.
\end{proof}

\begin{lem}\label{l:u3}
The conclusion to Theorem \ref{t:class_main} holds when $G = {\rm U}_3(q)$.
\end{lem}

\begin{proof}
This is very similar to the proof of Lemma \ref{l:l3}. Note that $q \geqs 3$ (since ${\rm U}_3(2)$ is soluble) and it suffices to show that
\begin{equation}\label{e:u3_prob}
\what{\b}(G) < \frac{q^3}{q^3+1}.
\end{equation}
Fix a prime $p \in \pi'(G)$ and let $x \in G$ be an element of order $p$.

Suppose $p=2$, so $q$ is odd and $|x^G| = q^2(q^2-q+1) = b_1$ for every involution $x \in G$. If $q \equiv 3 \imod{4}$ then $|H_2| = 2((q+1)_2)^2$ and $H_2$ is contained in a subgroup $L<G$ of type ${\rm GU}_1(q) \wr S_3$. By counting in $L$, we see that there are no more than $3q+6= a_1$ such involutions in $H_2$. On the other hand, if $q \equiv 1 \imod{4}$ then $|H_2| \leqs 4(q-1) = a_1$. It follows that $\rho_0 \leqs a_1^2b_1^{-1}$ for all odd $q$.

Now assume $p$ is odd, so $p \in \mathcal{P}_m$ with $m \in \{1,2,6\}$. Set $\gamma = (4,q-1)$ and $\delta = (4,q+1)$. If $m=6$ then 
\[
|H_p| = (q^2-q+1)_p \leqs q^2-q+1 = a_2,\;\; |x^G| = q^3(q+1)(q^2-1) = b_2
\]
and thus $\rho_3 \leqs a_2^2b_2^{-1}\log a_2$. Similarly, $\rho_1 \leqs a_3^2b_3^{-1}\log a_3$, where $a_3 = (q-1)/\gamma$ and $b_3 = q^3(q^3+1)$. Finally, suppose $m=2$ and note that 
\[
|H_p| = ((q+1)_p)^2(3!)_p \leqs 3((q+1)/\delta)^2 = a_5,
\]
which means that $H_p$ is contained in a subgroup of type ${\rm GU}_1(q) \wr S_3$. If $\nu(x) = 1$ then $|x^G| = q^2(q^2-q+1) = b_4$ and there are at most $3((q+1)/\delta-1) = a_4$ such elements in $H_p$. Otherwise, $|x^G| \geqs \frac{1}{3}q^3(q-1)(q^2-q+1) = b_5$ and we deduce that 
\[
\rho_2 \leqs \log((q+1)/\delta) \cdot (a_4^2b_4^{-1} + a_5^2b_5^{-1}).
\] 

The above bounds imply that \eqref{e:u3_prob} holds if $q \geqs 9$. And for $q \leqs 8$ we can use {\sc Magma} to show that $\Sigma(G)<1$ in every case. The result follows. 
\end{proof}

\subsection{Symplectic groups}\label{ss:symp}

In this section we prove Theorem \ref{t:class_main} for the symplectic groups. Recall that we may assume $G = {\rm PSp}_n(q)$, where $n \geqs 4$ is even and $(n,q) \ne (4,2)$ (the latter assumption is valid since ${\rm PSp}_4(2)' \cong {\rm L}_2(9)$). As before, we will work closely with the notation defined in \eqref{e:defs}, \eqref{e:pm}, \eqref{e:alpha} and \eqref{e:nu}.

Our main result is the following (the $4$-dimensional symplectic groups will be handled separately in Lemma \ref{l:sp4}).

\begin{lem}\label{l:spn}
The conclusion to Theorem \ref{t:class_main} holds when $G = {\rm PSp}_n(q)$ and $n \geqs 6$.
\end{lem}

\begin{proof}
Write $n = 2\ell$ and set $d = (2,q-1)$. We have 
\[
|G| = \frac{1}{d}q^{\ell^2}\prod_{i=1}^\ell(q^{2i}-1), \;\;  |N_G(H_r)| = \frac{1}{d}q^{\ell^2}(q-1)^{\ell}
\]
and thus $\Sigma(G)<1$ if 
\begin{equation}\label{e:spn_prob}
\what{\b}(G) \leqs \left(1-q^{-1}\right)^{\ell}.
\end{equation}
We divide the proof into a number of separate cases.

\vs

\noindent \emph{Case 1. $n \geqs 10$}

\vs

To begin with, we assume $n \geqs 10$. Fix a prime $p \in \pi'(G)$ and let $x \in G$ be an element of order $p$. For now, let us assume $p=2$, so $q$ is odd and either $\nu(x)$ is even or $\nu(x) = \ell$. By appealing to Lemma \ref{l:nt}(i), we compute 
\[
|H_2| = \frac{1}{2}((q^2-1)_2)^\ell(\ell!)_2 \leqs 2^{n-1}(q+1)^\ell = a_2
\]
and thus $H_2$ is contained in a subgroup $L$ of type ${\rm Sp}_2(q) \wr S_\ell$. If $\nu(x) = 2$ then 
\[
|x^G| = \frac{|{\rm Sp}_n(q)|}{|{\rm Sp}_{n-2}(q)||{\rm Sp}_2(q)|} = \frac{q^{n-2}(q^n-1)}{q^2-1} = b_1
\]
and by counting in $L$ we deduce that there are no more than 
\[
a_1 = \binom{\ell}{1} + \binom{\ell}{2}|{\rm Sp}_2(q)| = \ell+\frac{1}{2}\ell(\ell-1)q(q^2-1)
\]
of these involutions in $H_2$. And if $\nu(x) \geqs 4$ then 
\[
|x^G| \geqs \frac{|{\rm Sp}_n(q)|}{|{\rm Sp}_{n-4}(q)||{\rm Sp}_4(q)|}>\frac{1}{2}q^{4n-16} = b_2
\]
and thus Lemma \ref{l:est} implies that $\rho_0 < a_1^2b_1^{-1} + a_2^2b_2^{-1}$.

Now assume $p$ is odd, so $p \in \mathcal{P}_{m}$ with $m \leqs n/e$, where $e=2$ if $m$ is odd, otherwise $e=1$. Set $k = \lfloor n/me\rfloor$. First assume 
$m=4$, in which case $p$ divides $q^2+1$ and  
\[
|x^G| \geqs \frac{|{\rm Sp}_n(q)|}{|{\rm Sp}_{n-4}(q)||{\rm GU}_1(q^2)|}>\frac{1}{2}\left(\frac{q}{q+1}\right)q^{4n-8} = b_3.
\]
In addition, by appealing to Lemma \ref{l:factorial} (noting that $p \geqs 5$), we deduce that 
\[
|H_p| = ((q^4-1)_p)^k(k!)_p \leqs (q^{2}+1)^{k}\cdot p^{\frac{k}{4}} \leqs (q^2+1)^{\frac{5}{16}n} = a_3.
\]
Now suppose $m \geqs 3$ and $m \ne 4$. Here $|H_p| = ((q^{m}-1)_p)^k(k!)_p$
and
\[
|x^G| \geqs \frac{|{\rm Sp}_n(q)|}{|{\rm Sp}_{n-6}(q)||{\rm GU}_1(q^3)|}>\frac{1}{2}\left(\frac{q}{q+1}\right)q^{6n-18} = b_4.
\]
If $m$ is even, then $p$ divides $q^{m/2}+1$ and we get
\[
|H_p| \leqs (q^{m/2}+1)^{k}\cdot p^{\frac{k}{4}} \leqs (q^{m/2}+1)^{\frac{5n}{4m}} \leqs (q^3+1)^{\frac{5}{24}n} = a_4.
\]
Similarly, if $m$ is odd then $|H_p| \leqs (q^{m}-1)^{5n/8m} < a_4$. 

Since there are at most $n-3$ possibilities for $m$ in the range $3 \leqs m \leqs n$ with $m \ne 4$, and Lemma \ref{l:primes} implies that $|\mathcal{P}_m| < \frac{1}{2}n\log q$ for all $m$, we deduce that  
\[
\rho_3 < \log(q^2+1) \cdot a_3^2b_3^{-1}+\frac{1}{2}n(n-3)\log q \cdot a_4^2b_4^{-1}.
\]

Next assume $m = 2$, so $p$ divides $q+1$ and $|H_p| = ((q+1)_p)^{\ell}(\ell!)_p$, which means that $H_p$ is contained in a subgroup $L$ of type ${\rm GU}_1(q) \wr S_{\ell}$. If $\nu(x) = 2$ then 
\[
|x^G| = \frac{|{\rm Sp}_n(q)|}{|{\rm Sp}_{n-2}(q)||{\rm GU}_1(q)|} = \frac{q^{n-1}(q^n-1)}{q+1} = b_5
\]
and by counting in $L$ we see that there are no more than 
\[
\binom{\ell}{1}(p-1) \leqs \frac{1}{2}nq = a_5
\]
such elements in $H_p$. Otherwise, we have $|x^G| >\frac{1}{2}(q+1)^{-1}q^{4n-9} = b_6$ and we observe that $|H_p| < (q+1)^{3n/4} = a_6$. This gives 
$\rho_2 < \log(q+1) \cdot (a_5^2b_5^{-1}+a_6^2b_6^{-1})$.

The case $m = 1$ is very similar. Here $|H_p| = ((q-1)_p)^{\ell}(\ell!)_p<q^{3n/4} = a_8$ and thus $H_p$ is contained in a subgroup of type ${\rm GL}_1(q) \wr S_{\ell}$. If $\nu(x) = 2$, then 
\[
|x^G| = \frac{|{\rm Sp}_n(q)|}{|{\rm Sp}_{n-2}(q)||{\rm GL}_1(q)|} = \frac{q^{n-1}(q^n-1)}{q-1} = b_7
\]
and there are no more than $a_7 = n(q-2)/2$ of these elements in $H_p$. On the other hand, if $\nu(x) \geqs 4$ then $|x^G| >\frac{1}{2}q^{4n-10} = b_8$ and thus 
$\rho_1 <  \log(q-1) \cdot (a_7^2b_7^{-1}+a_8^2b_8^{-1})$.

For each $i \in \{0,1,2,3\}$ it is clear that the explicit upper bound $\rho_i < f_i(n,q)$ presented above implies that $\rho_i \to 0$ as $n$ or $q$ tends to infinity. In particular, this means that $\b(G) \to 0$ as $|G| \to \infty$. In addition, one can check that if $q \geqs 3$ then the above estimates imply that the inequality in \eqref{e:spn_prob} is satisfied unless $(n,q) = (10,3)$. To handle the latter case we can set $a_2 = |H_2| = 2^{17}$ in the upper bound on $\rho_0$ and then it is easy to check that the previous bounds are sufficient.

Now assume $q = 2$. Here $\mathcal{P}_1 = \emptyset$, $\mathcal{P}_2 = \{3\}$ and $\mathcal{P}_4 = \{5\}$, which means that
\[
\what{\b}(G) < a_3^2b_3^{-1} + \frac{1}{2}n(n-3)a_4^2b_4^{-1} + a_5^2b_5^{-1} + a_6^2b_6^{-1},
\]
where the $a_i,b_i$ terms are defined as above. One checks that this upper bound is sufficient unless $n = 10$, in which case we can use {\sc Magma} to show that  $\Sigma(G)<1$.

\vs

\noindent \emph{Case 2. $n = 8$}

\vs

Fix a prime $p \in \pi'(G)$ and let $x \in G$ be an element of order $p$. 
For $p=2$ we have $|H_2| \leqs 2^6(q+1)^4 = a_2$ and 
\[
|x^G| \geqs \frac{|{\rm Sp}_8(q)|}{2|{\rm Sp}_4(q)|^2} > \frac{1}{4}q^{16} = b_2
\]
if $\nu(x) = 4$. So by arguing as in Case 1 we deduce that $\rho_0< a_1^2b_1^{-1}+a_2^2b_2^{-1}$, where $a_1 = 4+6q(q^2-1)$ and $b_1 = q^6(q^2+1)(q^4+1)$. 

Now assume $p \geqs 3$ is odd, so $p \in \mathcal{P}_m$ with $m \in \{1,2,3,4,6,8\}$. For $m = 4$ we have $|H_p| \leqs (q^2+1)^2 = a_3$ and $|x^G|>\frac{1}{2}(q+1)^{-1}q^{25} = b_3$. Similarly, if $m \in \{3,6,8\}$ then $|H_p| \leqs q^4+1 = a_4$ and $|x^G|>\frac{1}{2}(q+1)^{-1}q^{31} = b_4$, which allows us to conclude that
\[
\rho_3 < \log(q^2+1) \cdot a_3^2b_3^{-1} + 3\log(q^4+1) \cdot a_4^2b_4^{-1}.
\]
Next suppose $m=2$, in which case $|H_p| = ((q+1)_p)^4(4!)_p \leqs 3(q+1)^4 = a_6$. If $\nu(x) = 2$ then $|x^G| = q^7(q^8-1)/(q+1) = b_5$ and by arguing as in Case 1 we deduce that there are at most $a_5 = 4q$ such elements in $H_p$. And if $\nu(x) > 2$, then 
\[
|x^G| \geqs \frac{|{\rm Sp}_8(q)|}{|{\rm GU}_4(q)|} > \frac{1}{2}\left(\frac{q}{q+1}\right)q^{20} = b_6
\]
and thus $\rho_2 < \log(q+1) \cdot (a_5^2b_5^{-1} + a_6^2b_6^{-1})$.
Similarly, $\rho_1 <  \log(q-1) \cdot (a_7^2b_7^{-1} + a_8^2b_8^{-1})$, where
\[
a_7 = 4(q-2),\; a_8 = 3(q-1)^4,\; b_7 = \frac{q^{7}(q^8-1)}{q-1},\; b_8 = \frac{1}{2}q^{20}.
\]

Putting these estimates together, it is clear that $\b(G) \to 0$ as $q \to \infty$. Furthermore, we deduce that \eqref{e:spn_prob} holds (with $\ell=4$) if $q \geqs 4$. Finally, for $q \in \{2,3\}$ we can use {\sc Magma} to check that $\Sigma(G)<1$. 

\vs

\noindent \emph{Case 3. $n = 6$}

\vs

This is very similar to the previous case. As usual, fix a prime $p \in \pi'(G)$ and let $x \in G$ be an element of order $p$.

First assume $p=2$ and note that $|H_2| \leqs 2^3(q+1)^3 = a_2$. If $\nu(x) = 3$ then  
\[
|x^G| \geqs \frac{|{\rm Sp}_6(q)|}{2|{\rm GU}_3(q)|} > \frac{1}{4}\left(\frac{q}{q+1}\right)q^{12} = b_2,
\]
otherwise $|x^G| = q^4(q^4+q^2+1) = b_1$ and as in Case 1 we see that there are at most $a_1 = 3+3q(q^2-1)$ involutions $x \in H_2$ with $\nu(x) = 2$. Therefore, $\rho_0< a_1^2b_1^{-1}+a_2^2b_2^{-1}$. 

Now assume $p \geqs 3$ is odd, so $p \in \mathcal{P}_m$ with $m \in \{1,2,3,4,6\}$. If $m = 4$ then $|H_p| \leqs q^2+1 = a_3$ and $|x^G|>\frac{1}{2}(q+1)^{-1}q^{17} = b_3$. Similarly, if $m \in \{3,6\}$ then $|H_p| \leqs q^2+q+1 = a_4$ and $|x^G|>\frac{1}{2}(q+1)^{-1}q^{19} = b_4$, which allows us to conclude that
\[
\rho_3 < \log(q^2+1) \cdot a_3^2b_3^{-1} + 2\log(q^2+q+1) \cdot a_4^2b_4^{-1}.
\]
Next suppose $m=2$, in which case $|H_p| = ((q+1)_p)^3(3!)_p \leqs 3(q+1)^3 = a_6$. If $\nu(x) = 2$ then 
\[
|x^G| = \frac{|{\rm Sp}_6(q)|}{|{\rm Sp}_4(q)||{\rm GU}_1(q)|} = \frac{q^5(q^6-1)}{q+1} = b_5
\]
and by arguing as in Case 1 we deduce that there are at most $a_5 = 3q$ such elements in $H_p$. Otherwise, we have $|x^G| > \frac{1}{2}(q+1)^{-1}q^{13} = b_6$ and we conclude that 
\[
\rho_2 < \log(q+1)\cdot (a_5^2b_5^{-1} + a_6^2b_6^{-1}).
\]
Similarly, $\rho_1 <  \log(q-1) \cdot (a_7^2b_7^{-1} + a_8^2b_8^{-1})$, where
\[
a_7 = 3(q-2),\; a_8 = 3(q-1)^3,\; b_7 = \frac{q^5(q^6-1)}{q-1},\; b_8 = \frac{1}{2}q^{12}.
\]

It is clear that each upper bound $\rho_i < f_i(q)$ given above has the property that $f_i(q) \to 0$ as $q \to \infty$, whence $G$ satisfies the asymptotic statement in Theorem \ref{t:class_main}. In addition, one checks that the given estimates imply that 
$\Sigma(G)<1$ if $q \geqs 4$. Finally, for $q \in \{2,3\}$ we can use {\sc Magma} to show that 
$\Sigma(G)<1$. 
\end{proof}

\begin{lem}\label{l:sp4}
The conclusion to Theorem \ref{t:class_main} holds when $G = {\rm PSp}_4(q)$.
\end{lem}

\begin{proof}
We may assume $q \geqs 4$ since ${\rm PSp}_{4}(2)' \cong {\rm L}_2(9)$ and ${\rm PSp}_4(3) \cong {\rm U}_4(2)$. Note that we have $\Sigma(G) < 1$ if
\begin{equation}\label{e:sp4_prob}
\what{\b}(G) < \frac{q^4}{(q+1)^2(q^2+1)}.
\end{equation}
Fix a prime $p \in \pi'(G)$ and let $x \in G$ be an element of order $p$.

Suppose $p=2$, so $q$ is odd and $G$ has two conjugacy classes of involutions; in the notation of \cite[Table 4.5.1]{GLS}, $x$ is either a type $t_1$ involution, or $x$ is of type $t_2$ (if $q \equiv 1 \imod{4}$) or $t_2'$ (if $q \equiv 3 \imod{4}$). If $x$ is a $t_1$ involution, then
\[
|x^G| = \frac{|{\rm Sp}_4(q)|}{2|{\rm Sp}_2(q)|^2} = \frac{1}{2}q^2(q^2+1) = b_1
\]
and the proof of \cite[Proposition 3.6]{BH} shows that there are at most $a_1 = q+2$ such involutions in $H_2$. On the other hand, if $x$ is a $t_2$ or $t_2'$ involution then 
\[
|x^G| \geqs \frac{|{\rm Sp}_4(q)|}{2|{\rm GU}_2(q)|} = \frac{1}{2}q^3(q-1)(q^2+1) = b_2
\]
and we note that $|H_2| \leqs 4(q+1)^2 = a_2$. This implies that $\rho_0 \leqs a_1^2b_1^{-1} + a_2^2b_2^{-1}$. 

For the remainder, we may assume $p$ is odd, in which case $p \in \mathcal{P}_m$ with $m \in \{1,2,4\}$. If $m=4$ then $p$ divides $q^2+1$ and we have
$|H_p| \leqs q^2+1 = a_3$ and 
\[
|x^G| = \frac{|{\rm Sp}_4(q)|}{|{\rm GU}_1(q^2)|} = q^4(q^2-1)^2 = b_3,
\]
whence $\rho_3 \leqs \log(q^2+1) \cdot a_3^2b_3^{-1}$. Similarly, if $m=2$ then we have $|H_p| \leqs (q+1)^2 = a_4$ and $|x^G| \geqs q^3(q-1)(q^2+1) = b_4$, while for $m=1$ we see that $|H_p| \leqs (q-1)^2 = a_5$ and $|x^G| \geqs q^3(q+1)(q^2+1) = b_5$. Therefore,
\[
\rho_2 \leqs \log(q+1) \cdot a_4^2b_4^{-1},\;\; \rho_1 \leqs \log(q-1) \cdot a_5^2b_5^{-1}.
\]

Bringing these estimates together, it is clear that $\b(G) \to 0$ as $q \to \infty$. In addition, we deduce that the inequality in \eqref{e:sp4_prob} holds if $q \geqs 9$. Finally, for $q \in \{4,5,7,8\}$ we can use {\sc Magma} to verify the bound $\Sigma(G)<1$.  
\end{proof}

\subsection{Orthogonal groups}\label{ss:orth}

To complete the proof of Theorem \ref{t:class_main}, we may assume $G$ is an orthogonal group. As before, we continue to work with the notation defined in \eqref{e:defs}, \eqref{e:pm}, \eqref{e:alpha} and \eqref{e:nu}.

We begin by handling the odd-dimensional groups of the form $G = \O_n(q)$, where $n \geqs 7$ is odd and $q$ is odd. 

\begin{lem}\label{l:on_odd}
The conclusion to Theorem \ref{t:class_main} holds when $G = \O_n(q)$ and $n \geqs 7$.
\end{lem}

\begin{proof}
Write $n = 2\ell+1$ and note that $q$ is odd. It is easy to check that $\Sigma(G)<1$ if \eqref{e:spn_prob} holds, so our main aim is to verify the bound in \eqref{e:spn_prob}. 
Fix a prime $p \in \pi'(G)$ and let $x \in G$ be an element of order $p$.

Suppose $p=2$ and write $q \equiv \e \imod{4}$ with $\e = \pm 1$. Then
\[
|H_2| = 2^{\ell-1}((q-\e)_2)^{\ell}(\ell!)_2 \leqs 2^{2\ell-1}(q+1)^{\ell} = a_3
\]
and we note that $H_2 < L < G$, where $L$ is of type $({\rm O}_2^{\e}(q) \wr S_{\ell}) \perp {\rm O}_1(q)$. If $\nu(x) = 1$ then 
\[
|x^G| \geqs \frac{|{\rm SO}_{n}(q)|}{2|{\rm SO}_{n-1}^{-}(q)|} = \frac{1}{2}q^{\ell}(q^{\ell}-1) = b_1
\]
and by counting in $L$ we deduce that there are at most 
\[
1+\binom{\ell}{1}(q-\e) \leqs 1+\ell(q+1) = a_1
\]
such involutions in $H_2$. Similarly, if $\nu(x) = 2$ then
\[
|x^G| \geqs \frac{|{\rm SO}_{n}(q)|}{2|{\rm SO}_{n-2}(q)||{\rm SO}_2^{-}(q)|} = \frac{q^{n-2}(q^{n-1}-1)}{2(q+1)} = b_2
\]
and we see that $H_2$ contains no more than
\[
\binom{\ell}{1}(1+q-\e) + \binom{\ell}{2}(|{\rm O}_2^{\e}(q)| + (q-\e)^2) \leqs \ell(q+2) + \frac{1}{2}\ell(\ell-1)(q+1)(q+3) = a_2
\]
such elements. Finally, if $\nu(x) \geqs 3$ then 
\[
|x^G| \geqs \frac{|{\rm SO}_n(q)|}{2|{\rm SO}_{n-3}^{-}(q)||{\rm SO}_3(q)|} = \frac{q^{\frac{1}{2}(3n-9)}(q^{\frac{1}{2}(n-3)}-1)(q^{n-1}-1)}{2(q^2-1)} = b_3
\]
and we conclude that 
\[
\rho_0 \leqs a_1^2b_1^{-1}+a_2^2b_2^{-1}+a_3^2b_3^{-1}.
\]

Now assume $p$ is odd and write $p \in \mathcal{P}_m$. Note that $m \leqs (n-1)/e$, where $e = 2$ if $m$ is odd, otherwise $e = 1$. Set $k = \lfloor (n-1)/me \rfloor$. If $m=4$ then $p$ divides $(q^2+1)/2$ and we have 
\[
|H_p| = ((q^4-1)_p)^k(k!)_p \leqs ((q^2+1)/2)^{\frac{5}{16}(n-1)}=a_4
\]
and
\[
|x^G| \geqs \frac{|{\rm SO}_n(q)|}{|{\rm SO}_{n-4}(q)||{\rm GU}_1(q^2)|} > \frac{1}{2}\left(\frac{q}{q+1}\right)q^{4n-12} = b_4.
\]
Similarly, if $m \geqs 3$ and $m \ne 4$, then $|\mathcal{P}_m|<\frac{1}{2}(n-1)\log q$ (see Lemma \ref{l:primes}) and we compute
\[
|H_p| = ((q^{m}-1)_p)^k(k!)_p \leqs ((q^3+1)/2)^{\frac{5}{24}(n-1)}=a_5
\]
and
\[
|x^G| \geqs \frac{|{\rm SO}_n(q)|}{|{\rm SO}_{n-6}(q)||{\rm GU}_1(q^3)|} > \frac{1}{2}\left(\frac{q}{q+1}\right)q^{6n-24}=b_5.
\]
This implies that
\[
\rho_3 < \log(q^2+1) \cdot a_4^2b_4^{-1} + \frac{1}{2}(n-1)(n-3)\log q \cdot a_5^2b_5^{-1}.
\]

Now assume $m=2$, so $p$ divides $(q+1)/2$ and we have $k = \ell$. Here
\[
|H_p| = ((q+1)_p)^{\ell}(\ell!)_p \leqs ((q+1)/2)^{3\ell/2} = a_7
\]
and we see that $H_p < L < G$, where $L$ is a subgroup of type $({\rm O}_2^{-}(q) \wr S_{\ell}) \perp {\rm O}_1(q)$.  If $\nu(x) = 2$ then 
\[
|x^G| = \frac{|{\rm SO}_n(q)|}{|{\rm SO}_{n-2}(q)||{\rm GU}_1(q)|} = \frac{q^{n-2}(q^{n-1}-1)}{q+1} = b_6
\]
and by counting in $L$ we deduce that there are at most $a_6 = \ell q$ such elements in $H_p$. On the other hand, if $\nu(x)>2$ then 
\[
|x^G| \geqs \frac{|{\rm SO}_n(q)|}{|{\rm SO}_{n-4}(q)||{\rm GU}_2(q)|}>\frac{1}{2}\left(\frac{q}{q+1}\right)q^{4n-14} = b_7
\]
and it follows that 
\[
\rho_2 < \log((q+1)/2) \cdot (a_6^2b_6^{-1}+a_7^2b_7^{-1}).
\]
Similarly,  we have $\rho_1 < \log((q-1)/2) \cdot (a_8^2b_8^{-1}+a_9^2b_9^{-1})$ with
\[
a_8 = \ell(q-2),\; a_9 = ((q-1)/2)^{3\ell/2},\; b_8 = \frac{q^{n-2}(q^{n-1}-1)}{q-1},\; b_9 = \frac{1}{2}q^{4n-14}.
\]

It is clear to see that the above estimates imply that $\b(G) \to 0$ as $|G| \to \infty$. We also deduce that the inequality in \eqref{e:spn_prob} holds if $q \geqs 5$, or if $q = 3$ and $n \geqs 21$. If $q = 3$ and $11 \leqs n \leqs 19$ then by setting $a_3 = |H_2|$ it is easy to check that the above bounds are sufficient. Finally, for $(n,q) = (9,3), (7,3)$ we can use {\sc Magma} to verify the bound $\Sigma(G)<1$.
\end{proof}

\begin{lem}\label{l:on_evenplus}
The conclusion to Theorem \ref{t:class_main} holds when $G = {\rm P\O}_n^{+}(q)$ and $n \geqs 10$.
\end{lem}

\begin{proof}
Write $n = 2\ell$ and set $d = (4,q^{\ell}-1)$. Since $|N_G(H_r)| = \frac{1}{d}q^{\ell(\ell-1)}(q-1)^{\ell}$, it follows that $\Sigma(G)<1$ if \eqref{e:spn_prob} holds. Fix a prime $p \in \pi'(G)$ and let $x \in G$ be an element of order $p$.

Suppose $p=2$, so $q$ is odd and either $\nu(x)$ is even or $\nu(x) = \ell$. Note that if $\nu(x) = 2$ then
\[
|x^G| \geqs \frac{|{\rm SO}_{n}^{+}(q)|}{2|{\rm SO}_{n-2}^{-}(q)||{\rm SO}_2^{-}(q)|} = \frac{q^{n-2}(q^{(n-2)/2}-1)(q^{n/2}-1)}{2(q+1)} = b_1,
\]
otherwise we have
\[
|x^G| \geqs \frac{|{\rm SO}_{n}^{+}(q)|}{2|{\rm SO}_{n-4}^{-}(q)||{\rm SO}_4^{-}(q)|} > \frac{1}{4}q^{4n-16} = b_2
\]
if $n \geqs 14$, and 
\[
|x^G| \geqs \frac{|{\rm SO}_{n}^{+}(q)|}{2|{\rm GU}_{n/2}(q)|} > \frac{1}{4}\left(\frac{q}{q+1}\right)q^{\frac{1}{4}(n^2-2n)} = b_2
\]
if $n \in \{10,12\}$.

First assume $q \equiv 1 \imod{4}$. Then by applying Lemma \ref{l:nt}(i), we compute
\[
|H_2| = \frac{1}{4}2^{\ell-1}((q-1)_2)^{\ell}(\ell!)_2 \leqs 2^{n-3}(q-1)^{\ell} = a_2
\]
and $H_2 < L<G$ with $L$ of type ${\rm O}_{2}^{+}(q) \wr S_{\ell}$. By counting in $L$ we deduce that there are at most
\[
a_1 = \ell + \binom{\ell}{2}\left((q-1)^2+2(q-1)\right)
\]
involutions $x \in H_2$ with $\nu(x) = 2$.

Now assume $q \equiv 3 \imod{4}$. If $\ell$ is even then   
\[
|H_2| = \frac{1}{4}2^{\ell-1}((q+1)_2)^{\ell}(\ell!)_2 \leqs 2^{n-3}(q+1)^{\ell} = a_2
\]
and $H_2 < L<G$ with $L$ of type ${\rm O}_{2}^{-}(q) \wr S_{\ell}$. And by working in $L$ we calculate that $H_2$ contains no more than
\[
a_1 = \ell + \binom{\ell}{2}\left((q+1)^2+2(q+1)\right)
\]
involutions with $\nu(x) = 2$. Similarly, if $\ell$ is odd then 
\[
|H_2|  = \frac{1}{2}2^{\ell-1}((q+1)_2)^{\ell-1}((\ell-1)!)_2.2 \leqs 2^{n-2}(q+1)^{\ell-1} = a_2
\]
and $H_2$ is contained in a subgroup $L$ of type $({\rm O}_2^{-}(q) \wr S_{\ell-1})\perp {\rm O}_2^{+}(q)$. By counting in $L$, we calculate that the number of involutions $x \in H_2$ with $\nu(x) = 2$ is at most
\[
a_1 = 1+(\ell-1)((q^2-1)+1)+\binom{\ell-1}{2}\left((q+1)^2+2(q+1)\right).
\]

So for all odd $q$ we deduce that $\rho_0 < a_1^2b_1^{-1} + a_2^2b_2^{-1}$, where the precise expressions for $a_1$ and $a_2$ depend on the congruence classes of $n$ and $q$ modulo $4$, as defined above.

Now assume $p \geqs 3$ is odd, so $p \in \mathcal{P}_m$ for some $m \leqs n-2$. Set $k = \lfloor (n-2)/me \rfloor$, where $e=2$ if $m$ is odd, otherwise $e = 1$. If $m=4$ then $p$ divides $q^2+1$ and we have
\[
|x^G| \geqs \frac{|{\rm SO}_n^{+}(q)|}{|{\rm SO}_{n-4}^{-}(q)||{\rm GU}_1(q^2)|} > \frac{1}{2}\left(\frac{q}{q+1}\right)q^{4n-12} = b_3.
\]
In addition, since $p \geqs 5$ and $(q^{n/2}-1)_p \leqs \frac{1}{8}n(q^2+1)$ by Lemma \ref{l:nt}, we compute
\begin{align*}
|H_p| = ((q^4-1)_p)^k(k!)_p \cdot (q^{n/2}-1)_p & \leqs (q^2+1)^{\frac{5}{16}(n-2)} \cdot \frac{1}{8}n(q^2+1) \\
&  = \frac{1}{8}n(q^2+1)^{\frac{1}{16}(5n+6)} = a_3.
\end{align*}
Similarly, if $m \geqs 3$ and $m \ne 4$, then
\[
|x^G| \geqs \frac{|{\rm SO}_n^{+}(q)|}{|{\rm SO}_{n-6}^{-}(q)||{\rm GU}_1(q^3)|} > \frac{1}{2}\left(\frac{q}{q+1}\right)q^{6n-24} = b_4
\]
and
\[
|H_p| = ((q^m-1)_p)^k(k!)_p \cdot (q^{n/2}-1)_p.
\]
If $m \geqs 6$ is even, then $p$ divides $q^{m/2}+1$ and thus
\[
|H_p| < (q^{m/2}+1)^{\frac{5(n-2)}{4m}}\cdot q^{\frac{1}{2}n} \leqs (q^3+1)^{\frac{5}{24}(n-2)} \cdot q^{\frac{1}{2}n} = a_4,
\]
while we get 
\[
|H_p| < q^{\frac{5}{4}mk + \frac{1}{2}n} = q^{\frac{5}{8}(n-2)+\frac{1}{2}n} < a_4
\]
if $m \geqs 3$ is odd. This allows us to conclude that
\[
\rho_3 < \log(q^2+1) \cdot a_3^2b_3^{-1} + \frac{1}{2}n(n-5)\log q \cdot a_4^2b_4^{-1}.
\]

Next assume $m = 2$, so $p$ divides $q+1$. First observe that if $\ell$ is even then
\[
|H_p| = ((q+1)_p)^{\ell}(\ell!)_p \leqs (q+1)^{\frac{3}{4}n} = a_6
\]
and $H_p < L<G$ with $L$ of type ${\rm O}_2^{-}(q) \wr S_{\ell}$. Similarly, if $\ell$ is odd then 
\[
|H_p| = ((q+1)_p)^{\ell-1}((\ell-1)!)_p < a_6
\]
and $H_p$ is contained in a subgroup $L$ of type $({\rm O}_2^{-}(q) \wr S_{\ell-1})\perp {\rm O}_2^{+}(q)$. If $\nu(x) = 2$ then 
\[
|x^G| = \frac{|{\rm SO}_{n}^{+}(q)|}{|{\rm SO}_{n-2}^{-}(q)||{\rm GU}_1(q)|} = \frac{q^{n-2}(q^{(n-2)/2}-1)(q^{n/2}-1)}{q+1} = b_5,
\]
otherwise we have
\[
|x^G| \geqs \frac{|{\rm SO}_{n}^{+}(q)|}{|{\rm SO}_{n-4}^{+}(q)||{\rm GU}_2(q)|} > \frac{1}{2}\left(\frac{q}{q+1}\right)q^{4n-14} = b_6.
\]
And by counting in $L$, we deduce that there are at most $a_5 = \ell q$ elements $x \in H_p$ of order $p$ with $\nu(x) = 2$. It follows that 
$\rho_2 < \log(q+1) \cdot (a_5^2b_5^{-1} + a_6^2b_6^{-1})$.

Finally, suppose $m=1$. Here 
\[
|H_p| = ((q-1)_p)^{\ell}(\ell!)_p \leqs (q-1)^{\frac{3}{4}n} = a_8
\]
and $H_p < L<G$ with $L$ of type ${\rm O}_2^{+}(q) \wr S_{\ell}$. If $\nu(x) = 2$ then 
\[
|x^G| = \frac{|{\rm SO}_{n}^{+}(q)|}{|{\rm SO}_{n-2}^{+}(q)||{\rm GL}_1(q)|} = \frac{q^{n-2}(q^{(n-2)/2}+1)(q^{n/2}-1)}{q-1} = b_7
\]
and by working in $L$ we see that there are no more than $a_7 = \ell(q-2)$ of these elements in $H_p$. And for $\nu(x)>2$ 
we have $|x^G|>\frac{1}{2}q^{4n-14} = b_8$, so $\rho_1 < \log(q-1) \cdot (a_7^2b_7^{-1} + a_8^2b_8^{-1})$.

It is easy to see that each of the above bounds $\rho_i  < f_i(n,q)$ has the property that $f_i(n,q) \to 0$ as $n$ or $q$ tends to infinity. In particular, this implies that $\b(G) \to 0$ as $|G| \to \infty$. In addition, the above estimates imply that $\Sigma(G)<1$ unless $(n,q)$ is one of the following:
\[
(10,2), \; (10,3), \; (12,2), \; (12,3), \; (14,2).
\]
For $(n,q) = (10,2)$ we can use {\sc Magma} to compute
\[
\Sigma(G)  = \frac{11219267305357}{11749647974400} < 1
\] 
as required. And in each of the remaining cases, we can slightly modify our previous estimates in order to obtain the desired conclusion. For example, suppose $(n,q) = (10,3)$. Here we compute $i_2(H_2) = 2063$ and $\pi'(G) = \{5,7,11,13,41\}$, so we can set $a_2 = 2063$  in the upper bound on $\rho_0$, and we can also replace the term $\frac{1}{2}n(n-5)\log q$ in the upper bound on $\rho_3$ by $4$. One can then check that the revised estimates imply that $\Sigma(G)<1$.
\end{proof}

\begin{lem}\label{l:o8_plus}
The conclusion to Theorem \ref{t:class_main} holds when $G = {\rm P\O}_8^{+}(q)$.
\end{lem}

\begin{proof}
As in the proof of the previous lemma, we observe that $\Sigma(G)<1$ if the inequality in \eqref{e:spn_prob} holds (with $\ell=4$). For $q=3$, we can directly verify the bound $\Sigma(G)<1$ with the aid of {\sc Magma}. However, for $q=2$ we get 
\[
\Sigma(G) = \frac{6518753}{6220800} > 1.
\]
To resolve the latter case, we use {\sc Magma} to compute
\[
Q_p(G) = 1 - \frac{s|H_p|}{|G:H_p|}
\]
for each $p \in \pi(G) = \{2,3,5,7\}$, where $s$ is the number of $(H_p,H_p)$ double cosets in $G$ of size $|H_p|^2$. In this way, we obtain
\[
Q_2(G) = \frac{38429}{42525}, \; Q_3(G) = \frac{4307}{89600}, \; Q_5(G) = \frac{31}{435456}, \; Q_7(G) = \frac{1}{4147200} 
\]
and thus $\a(G) < 1$ as required. 

For the remainder, we may assume $q \geqs 4$. As usual, fix a prime $p \in \pi'(G)$ and let $x \in G$ be an element of order $p$.

If $p=2$ then $|H_2| \leqs 16(q+1)^4 = a_1$ and 
\[
|x^G| \geqs \frac{|{\rm SO}_8^{+}(q)|}{2|{\rm GU}_4(q)|} = \frac{1}{2}q^6(q-1)(q^2+1)(q^3-1) = b_1,
\]
so we have $\rho_0 \leqs a_1^2b_1^{-1}$.

Now assume $p$ is odd, so $p \in \mathcal{P}_m$ with $m \in \{1,2,3,4,6\}$. If $m \in \{3,4,6\}$ then we have $|H_p| \leqs (q^2+1)^2 = a_2$ and
\[
|x^G| \geqs \frac{|{\rm SO}_8^{+}(q)|}{|{\rm GU}_2(q^2)|} = q^{10}(q^2-1)^2(q^6-1) = b_2,
\]
which gives $\rho_3 \leqs 3\log(q^2+q+1) \cdot a_2^2b_2^{-1}$. Similarly, if $m=2$ then $|H_p| \leqs 3(q+1)^4 = a_3$ and
\[
|x^G| \geqs \frac{|{\rm SO}_8^{+}(q)|}{|{\rm GU}_4(q)|} = q^{6}(q-1)(q^2+1)(q^3-1) = b_3,
\]
whereas $|H_p| \leqs 3(q-1)^4 = a_4$ and 
\[
|x^G| \geqs \frac{|{\rm SO}_8^{+}(q)|}{|{\rm GL}_4(q)|} = q^{6}(q+1)(q^2+1)(q^3+1) = b_4
\]
if $m=1$. It follows that 
\[
\rho_2 \leqs \log(q+1) \cdot a_3^2b_3^{-1},\;\; \rho_1 \leqs \log(q-1) \cdot a_4^2b_4^{-1}.
\]

It is clear from the above estimates that $\b(G) \to 0$ as $q \to \infty$. Furthermore, we deduce that \eqref{e:spn_prob} holds (with $\ell=4$) if $q \geqs 8$, so it just remains to handle the groups with $q \in \{4,5,7\}$.

If $q = 7$ then using {\sc Magma} we compute $i_2(H_2) = 2831$ and by setting $a_1 = 2831$ in the previous estimate for $\rho_0$ it is easy to check \eqref{e:spn_prob} holds. And similarly if $q = 5$, noting that $i_2(H_2) = 495$. Finally, suppose $q=4$. Here $\pi'(G) = \{3,5,7,13,17\}$ with $\mathcal{P}_1 = \{3\}$ and $\mathcal{P}_2 = \{5\}$, so by defining $a_i,b_i$ as above for $i \in \{2,3,4\}$ we get 
\[
\what{\b}(G) \leqs 3a_2^2b_2^{-1} + a_3^2b_3^{-1} + a_4^2b_4^{-1} < \left(\frac{3}{4}\right)^4
\]
and the result follows.
\end{proof}

\begin{lem}\label{l:on_evenminus}
The conclusion to Theorem \ref{t:class_main} holds when $G = {\rm P\O}_n^{-}(q)$ and $n \geqs 8$.
\end{lem}

\begin{proof}
Write $n = 2\ell$, $d = (4,q^{\ell}+1)$ and note that $|N_G(H_r):H_r| = \frac{1}{d}(q^2-1)(q-1)^{\ell-2}$, which means that $\Sigma(G)<1$ if the bound in \eqref{e:spn_prob} is satisfied. Fix a prime $p \in \pi'(G)$ and let $x \in G$ be an element of order $p$. To begin with, we will assume $n \geqs 10$. 

\vs

\noindent \emph{Case 1. $n \geqs 10$}

\vs

First assume $p=2$. If $\nu(x) = 2$ then 
\[
|x^G| \geqs \frac{|{\rm SO}_n^{-}(q)|}{2|{\rm SO}_{n-2}^{+}(q)||{\rm SO}_2^{-}(q)|} = \frac{q^{n-2}(q^{(n-2)/2}+1)(q^{n/2}+1)}{2(q+1)} = b_1,
\]
otherwise $|x^G|>b_2$, where $b_2 = \frac{1}{4}q^{4n-16}$ if $n \geqs 14$ and $b_2 = \frac{1}{4}(q+1)^{-1}q^{c}$ if $n \in \{10,12\}$, with $c = 21$ if $n=10$ and $c = 31$ if $n=12$.

Suppose $q \equiv 1 \imod{4}$. Then 
\[
|H_2| = \frac{1}{2}2^{\ell-1}((q-1)_2)^{\ell-1}((\ell-1)!)_2.2 \leqs 2^{n-2}(q-1)^{\ell-1} = a_2
\]
and we observe that $H_2$ is contained in a subgroup $L<G$ of type $({\rm O}_{2}^{+}(q) \wr S_{\ell-1}) \perp {\rm O}_2^{-}(q)$. By counting in $L$ we see that there are no more than 
\[
a_1 = 1+\binom{\ell-1}{1}\left(1+(q^2-1)\right) + \binom{\ell-1}{2}\left((q-1)^2 +2(q-1)\right)
\]
involutions $x \in H_2$ with $\nu(x) = 2$.

Now assume $q \equiv 3 \imod{4}$. If $\ell$ is even, then   
\[
|H_2| = \frac{1}{2}2^{\ell-1}((q+1)_2)^{\ell-1}((\ell-1)!)_2.2 \leqs 2^{n-2}(q+1)^{\ell-1} = a_2
\]
and thus $H_2 < L<G$ with $L$ of type $({\rm O}_{2}^{-}(q) \wr S_{\ell-1}) \perp {\rm O}_2^{+}(q)$. By working in $L$ we deduce that $H_2$ contains at most
\[
a_1 = 1+\binom{\ell-1}{1}\left(1+(q^2-1)\right) + \binom{\ell-1}{2}\left((q+1)^2 +2(q+1)\right)
\]
involutions $x$ with $\nu(x) = 2$. Similarly, if $\ell$ is odd then 
\[
|H_2|  = \frac{1}{4}2^{\ell-1}((q+1)_2)^{\ell}(\ell!)_2 \leqs 2^{n-3}(q+1)^{\ell} = a_2
\]
and $H_2$ is contained in a subgroup $L$ of type ${\rm O}_2^{-}(q) \wr S_{\ell}$, which allows us to see that the number of involutions $x \in H_2$ with $\nu(x) = 2$ is at most
\[
a_1 = \binom{\ell}{1} + \binom{\ell}{2}\left((q+1)^2+2(q+1)\right).
\]
So for all odd $q$, we conclude that $\rho_0 < a_1^2b_1^{-1} + a_2^2b_2^{-1}$, where the expressions for $a_1$ and $a_2$ depend on the congruence classes of $n$ and $q$ modulo $4$, as defined above.

Now assume $p \geqs 3$ is odd, so $p \in \mathcal{P}_m$ for some $m \leqs n$. Set $k = \lfloor (n-2)/me \rfloor$, where $e=2$ if $m$ is odd, otherwise $e = 1$. If $m=4$ then $p$ divides $q^2+1$ and
\[
|x^G| \geqs \frac{|{\rm SO}_n^{-}(q)|}{|{\rm SO}_{n-4}^{+}(q)||{\rm GU}_1(q^2)|} > \frac{1}{2}\left(\frac{q}{q+1}\right)q^{4n-12} = b_3.
\]
Since $p \geqs 5$ and $(q^{n/2}+1)_p \leqs \frac{1}{4}n(q^2+1)$ (see Lemma \ref{l:nt}), we compute
\[
|H_p| = ((q^4-1)_p)^k(k!)_p \cdot (q^{n/2}+1)_p \leqs \frac{1}{4}n(q^2+1)^{\frac{1}{16}(5n+6)} = a_3.
\]
Similarly, if $m \geqs 3$ and $m \ne 4$, then
$|x^G| > \frac{1}{2}(q+1)^{-1}q^{6n-23} = b_4$
and
\[
|H_p| = ((q^m-1)_p)^k(k!)_p \cdot (q^{n/2}+1)_p \leqs (q^3+1)^{\frac{1}{24}(9n-10)}  = a_4.
\]
This implies that
\[
\rho_3 < \log(q^2+1) \cdot a_3^2b_3^{-1} + \frac{1}{2}n(n-4)\log q \cdot a_4^2b_4^{-1}.
\]

Next suppose $m = 2$. If $\ell$ is odd then
\[
|H_p| = ((q+1)_p)^{\ell}(\ell!)_p \leqs (q+1)^{\frac{3}{4}n} = a_6
\]
and $H_p < L<G$ with $L$ of type ${\rm O}_2^{-}(q) \wr S_{\ell}$. On the other hand, if $\ell$ is even, then 
\[
|H_p| = ((q+1)_p)^{\ell-1}((\ell-1)!)_p < a_6
\]
and $H_p$ is contained in a subgroup $L$ of type $({\rm O}_2^{-}(q) \wr S_{\ell-1})\perp {\rm O}_2^{+}(q)$. If $\nu(x) = 2$ then 
\[
|x^G| = \frac{|{\rm SO}_{n}^{-}(q)|}{|{\rm SO}_{n-2}^{+}(q)||{\rm GU}_1(q)|} = \frac{q^{n-2}(q^{(n-2)/2}+1)(q^{n/2}+1)}{q+1} = b_5,
\] 
otherwise we have $|x^G| > \frac{1}{2}(q+1)^{-1}q^{4n-13} = b_6$. By counting in $L$, we calculate that there are at most $a_5 = \ell q$ elements $x \in H_p$ of order $p$ with $\nu(x) = 2$ and it follows that 
\[
\rho_2 < \log(q+1) \cdot (a_5^2b_5^{-1} + a_6^2b_6^{-1}).
\]

Finally, suppose $m=1$. Here 
\[
|H_p| = ((q-1)_p)^{\ell-1}((\ell-1)!)_p \leqs (q-1)^{\frac{3}{4}(n-2)} = a_8
\]
and $H_p < L<G$ with $L$ of type $({\rm O}_2^{+}(q) \wr S_{\ell-1}) \perp {\rm O}_2^{-}(q)$. If $\nu(x) = 2$ then 
\[
|x^G| = \frac{|{\rm SO}_{n}^{-}(q)|}{|{\rm SO}_{n-2}^{-}(q)||{\rm GL}_1(q)|} = \frac{q^{n-2}(q^{(n-2)/2}-1)(q^{n/2}+1)}{q-1} = b_7,
\] 
and by working in $L$ we see that there are no more than 
$a_7 = (\ell-1)(q-1)$ of these elements in $H_p$. Otherwise $|x^G|>\frac{1}{2}q^{4n-14} = b_8$ and we conclude that
\[
\rho_1 < \log(q-1) \cdot (a_7^2b_7^{-1} + a_8^2b_8^{-1}).
\]

The above estimates $\rho_i < f_i(n,q)$ immediately imply that $\b(G) \to 0$ as $|G| \to \infty$. Moreover, we deduce that $\Sigma(G)<1$ unless $(n,q)$ is one of the following:
\[
(10,2), \; (10,3), \; (12,2), \; (12,3), \; (14,2).
\]
For $(n,q) = (10,2)$ we can use {\sc Magma} to verify the bound $\Sigma(G)<1$, and in the remaining cases we can modify our previous estimates. For example, suppose $(n,q) = (12,2)$. Here $|\pi'(G)| = 7$, $|H_3| = 3^6$ and we see that $\mathcal{P}_1 = \emptyset$, $\mathcal{P}_2 = \{3\}$, $\mathcal{P}_4 = \{5\}$. Therefore,
\[
\what{\b}(G) < a_3^2b_3^{-1}+5a_4^2b_4^{-1} + a_5^2b_5^{-1}+a_6^2b_6^{-1} < 2^{-6},
\]
where $a_6 = 3^6$ and the remaining $a_i,b_i$ are defined as above. 

\vs

\noindent \emph{Case 2. $n=8$}

\vs

To complete the proof, we may assume $n=8$. We can use {\sc Magma} to show that $\Sigma(G)<1$ when $q = 2,3$, so we may assume $q \geqs 4$.
Fix a prime $p \in \pi'(G)$ and let $x \in G$ be an element of order $p$.

If $p=2$ then $|H_2| \leqs 16(q+1)^3 = a_1$ and 
\[
|x^G| \geqs \frac{|{\rm SO}_8^{-}(q)|}{2|{\rm SO}_6^{+}(q)||{\rm SO}_2^{-}(q)|} = \frac{1}{2}q^6(q^2-q+1)(q^4+1) = b_1,
\]
so we have $\rho_0 \leqs a_1^2b_1^{-1}$.

Now assume $p$ is odd, so $p \in \mathcal{P}_m$ with $m \in \{1,2,3,4,6,8\}$. If $m \geqs 3$ then $|H_p| \leqs q^4+1 = a_2$ and
\[
|x^G| \geqs \frac{|{\rm SO}_8^{-}(q)|}{|{\rm SO}_4^{+}(q)||{\rm GU}_1(q^2)|} = q^{10}(q^4+1)(q^6-1) = b_2,
\]
which gives $\rho_3 \leqs 4\log(q^4+1) \cdot a_2^2b_2^{-1}$. For $m=2$ we observe that $|H_p| \leqs 3(q+1)^3 = a_3$ and
\[
|x^G| \geqs \frac{|{\rm SO}_8^{-}(q)|}{|{\rm SO}_6^{+}(q)||{\rm GU}_1(q)|} = q^6(q^2-q+1)(q^4+1) = b_3,
\]
whereas $|H_p| \leqs 3(q-1)^3 = a_4$ and 
\[
|x^G| \geqs \frac{|{\rm SO}_8^{-}(q)|}{|{\rm SO}_{6}^{-}(q)||{\rm GL}_1(q)|} = q^{6}(q^2+q+1)(q^4+1) = b_4
\]
if $m=1$. It follows that
\[
\rho_2 \leqs \log(q+1) \cdot a_3^2b_3^{-1},\;\; \rho_1 \leqs \log(q-1) \cdot a_4^2b_4^{-1}.
\]

It is clear that the above estimates imply that $\b(G) \to 0$ as $q$ tends to infinity. Furthermore, we deduce that \eqref{e:spn_prob} holds (with $\ell=4$) for all $q \geqs 4$. 
\end{proof}

\vs

This completes the proof of Theorem \ref{t:class_main}. In Remark \ref{r:u42}(a) we checked that Conjecture \ref{c:ls} holds in the special case $G = {\rm U}_4(2)$, so by combining Theorems \ref{t:ex_main} and  \ref{t:class_main}, we deduce that the proofs of Theorems \ref{t:main1} and \ref{t:main_asymptotic} are complete for groups of Lie type. Similarly, in view of Remark \ref{r:u42}(b), we have also proved Theorem \ref{t:main2} for all groups of Lie type.

\section{Sporadic groups}\label{s:spor}

In this final section, we assume $G$ is a sporadic simple group and we complete the proofs of Theorems \ref{t:main1} and \ref{t:main2}. As before, we define
\[
\a(G) = \sum_{p \in \pi(G)} Q_p(G).
\]
Then our main result is the following.

\begin{thm}\label{t:spor}
Let $G$ be a finite simple sporadic group. Then $\a(G)<1$.
\end{thm}

\begin{proof}
%
First assume $G$ is one of the following groups:
\[
{\rm M}_{11}, \, {\rm M}_{12}, \, {\rm M}_{23}, \, {\rm J}_1, \, {\rm J}_2, \, {\rm J}_3, \, {\rm HS}, \, {\rm He}, \, {\rm McL}, \, {\rm Suz}, \, {\rm Ru}, \, {\rm Fi}_{22}, \, {\rm Fi}_{23}, \, {\rm Co}_1,\, {\rm Co}_3,\, {\rm O'N}.
\]
Here we use {\sc Magma} \cite{magma} and the function \texttt{AutomorphismGroupSimpleGroup} to construct $G$ as a permutation group and we construct a Sylow $p$-subgroup $H_p$ of $G$ for each prime $p \in \pi(G)$. As in the proof of Lemma \ref{l:ex_comp}, we can calculate $\widehat{Q}_p(G)$ for each $p \in \pi(G)$ and then it is easy to verify the inequality
\[
\sum_{p \in \pi(G)}\what{Q}_p(G) < 1.
\]

For $G \in \{ {\rm M}_{22}, {\rm M}_{24}, {\rm Co}_2 \}$ we  compute
\[
Q_2(G) = 1 - \frac{s|H_2|}{|G:H_2|}
\]
precisely, where $s$ is the number of $(H_2,H_2)$ double cosets in $G$ of size $|H_2|^2$, and then it is routine to check that 
\begin{equation}\label{e:est1}
Q_2(G) + \sum_{p \in \pi'(G)}\what{Q}_p(G) < 1,
\end{equation}
where $\pi'(G)$ is the set of odd prime divisors of $|G|$. 

Next suppose $G \in \{ {\rm HN}, {\rm J}_4, {\rm Ly}, {\rm Th}, {\rm Fi}_{24}'\}$. For each $p \in \pi(G)$ we can work with \textsf{GAP} \cite{GAP} and the \textsf{GAP} Character Table Library \cite{GAPCTL} to compute 
\[
a_{p,i} = \min\{|H_p|, |x_i^G \cap L| \,:\, L \in \mathcal{M}_p\}, \;\; b_{p,i} = |x_i^G|
\]
with respect to a complete set $\{x_1, \ldots, x_{k_p}\}$ of representatives of the conjugacy classes in $G$ of elements of order $p$ and a set $\mathcal{M}_p$ of representatives of the conjugacy classes of maximal overgroups of $H_p$ in $G$. To do this, first observe that the character table of $G$ is available in \cite{GAPCTL}. In addition, we can use the \texttt{Maxes} function to access the character table of each maximal subgroup $L$ of $G$, together with the corresponding fusion map from $L$-classes to $G$-classes. This allows us to calculate each term $a_{p,i}$ and $b_{p,i}$ as defined above, and then it is straightforward to check that 
\begin{equation}\label{e:est2}
\sum_{p \in \pi(G)} \what{Q}_p(G) \leqs \sum_{p \in \pi(G)} \sum_{i=1}^{k_p} a_{p,i}^2b_{p,i}^{-1} < 1.
\end{equation}

%
%

To complete the proof, we may assume $G = \mathbb{B}$ or $\mathbb{M}$.
By 
Lemma \ref{l:est} we have
\[
\sum_{p \in \pi'(G)} \widehat{Q}_p(G) < \sum_{p \in \pi'(G)} a_p^2b_p^{-1} < \eta(G),
\]
where $a_p = |H_p|$, $b_p = \min\{|x^G| \,:\, x \in G, \, |x| = p\}$ and we compute
\[
\eta(G) = \begin{cases}
2.4 \times 10^{-7} & \mbox{if $G = \mathbb{B}$} \\
5.7 \times 10^{-11} & \mbox{if $G = \mathbb{M}$,}
\end{cases}
\]
so it suffices to show that $\widehat{Q}_2(G) < 1-\eta(G)$.


The Baby Monster $G = \mathbb{B}$ has four classes of involutions, labelled $\texttt{2A}$, $\texttt{2B}$, $\texttt{2C}$ and $\texttt{2D}$. We can embed $H_2$ in a maximal subgroup $L = [2^{35}].(S_5 \times {\rm L}_3(2))$ and using \cite{GAPCTL} to access the corresponding character tables and fusion map, we compute
\[
\begin{array}{lll}
|\texttt{2A}| = 13571955000 = d_1 & & |\texttt{2A} \cap L| = 51512 = c_1 \\
|\texttt{2B}| = 11707448673375 = d_2 & & |\texttt{2B} \cap L| = 1172575 = c_2 \\
|\texttt{2C}| = 156849238149120000 = d_3 & & |\texttt{2C} \cap L| = 131022848 = c_3 \\
|\texttt{2D}| = 355438141723665000 = d_4 & & |\texttt{2D} \cap L| = 313463400 = c_4.
\end{array}
\]
This gives $\what{Q}_2(G) \leqs \sum_i c_i^2d_i^{-1} < 0.7< 1-\eta(G)$, as required. Similarly, the Monster $G = \mathbb{M}$ has two classes of involutions with
\[
|\texttt{2A}| = 97239461142009186000 = d_1,\;\; |\texttt{2B}| = 5791748068511982636944259375 = d_2.
\]
We can embed $H_2$ in a maximal subgroup $L = 2^{3+6+12+18}.(3S_6 \times {\rm L}_3(2))$ of $G$, noting that \cite[Proposition 3.9]{BOW} gives 
\[
|\texttt{2A} \cap L| = 3573456 = c_1,\;\; |\texttt{2B} \cap L| =  4026530095 = c_2.
\]
This yields $\what{Q}_2(G) \leqs c_1^2d_1^{-1} + c_2^2d_2^{-1} < 1.4 \times 10^{-7}$ and the result follows.
\end{proof}

\begin{rem}\label{r:spor0}
By arguing as in the proof of Theorem \ref{t:spor}, it is straightforward to show that $\a(G)<1$ for every almost simple sporadic group $G$. Therefore, Conjecture \ref{c:ls} holds in this slightly more general setting. In addition, we deduce that if $A$ and $B$ are nilpotent subgroups of $G$, then $A \cap B^x = 1$ for some $x \in G$. To verify the inequality $\a(G)<1$ for an almost simple sporadic group of the form $G = T.2$ with $T$ simple, we use \textsf{GAP} \cite{GAPCTL,GAP} to show that the inequality in \eqref{e:est2} holds for $G \in \{ {\rm HN}.2, {\rm Fi}_{24}'.2\}$. And in each of the remaining cases, we use {\sc Magma} \cite{magma} to check that the bound in \eqref{e:est1} is satisfied.  
\end{rem}

\begin{rem}\label{r:alt0}
Recall that Lisi and Sabatini have shown that the conclusion to Conjecture \ref{c:ls} holds for all alternating groups $G = A_n$ with $n \geqs N$, for some unspecified constant $N$. It remains an open problem to prove Conjecture \ref{c:ls} for all alternating groups, which would complete the proof of the conjecture for all simple groups, in view of Theorem \ref{t:main1}. For $G = A_n$ with $5 \leqs n \leqs 10$, we can use {\sc Magma} \cite{magma} to compute $Q_p(G)$ for every prime $p \in \pi(G)$ and this allows us to verify the bound $\sum_p Q_p(G) <1$. And for $11 \leqs n \leqs 50$ we can calculate $\widehat{Q}_p(G)$ by proceeding as in the proof of Lemma \ref{l:ex_comp} and we find that $\sum_p \widehat{Q}_p(G) <1$ in every case. This means that in order to complete the proof of Conjecture \ref{c:ls}, we may assume $G = A_n$ is an alternating group with $n >50$.
\end{rem}

\vs

By combining Theorem \ref{t:spor} with Theorems \ref{t:ex_main} and \ref{t:class_main}, and referring to Remark \ref{r:u42} for the special case $G = {\rm U}_4(2)$, we deduce that Theorems \ref{t:main1} and \ref{t:main2} hold. Then Corollary \ref{c:main1} follows by combining Theorem \ref{t:main1} with the main theorem of \cite{Kurm}, and we note that Corollary \ref{c:main2} is an immediate consequence. In view of Theorem \ref{t:deb}, Corollary \ref{c:main_asymptotic} follows immediately from Theorem \ref{t:main_asymptotic}. Finally, we combine Theorem \ref{t:main2} with the main theorem of \cite{Z_sym} to deduce Corollary \ref{c:main3}.

\end{document}